\documentclass{aims}
\usepackage{amsmath, dsfont, mathtools, cite}
\mathtoolsset{showonlyrefs=true,showmanualtags=true}
\usepackage{paralist}
\usepackage{graphics} 
\usepackage{epsfig} 
\usepackage{graphicx}  \usepackage{epstopdf}
\usepackage[colorlinks=true]{hyperref}
\hypersetup{urlcolor=blue, citecolor=red}

\def\currentvolume{X}
\def\currentissue{X}
\def\currentyear{200X}
\def\currentmonth{XX}
\def\ppages{X--XX}

\newtheorem{theorem}{Theorem}[section]
\newtheorem{corollary}[theorem]{Corollary}

\newtheorem{lemma}[theorem]{Lemma}
\newtheorem{assumption}[theorem]{Assumption}

\theoremstyle{definition}
\newtheorem{definition}[theorem]{Definition}
\newtheorem{remark}[theorem]{Remark}

\title[Infinite-dimensional Event/Self-triggered Control] 
      {Stability Analysis of Infinite-dimensional 
      	Event-triggered and Self-triggered Control Systems with Lipschitz Perturbations}

\author[Masashi Wakaiki and Hideki Sano]{}

\subjclass{Primary: 93C25, 93C57, 93D15; Secondary: 47D06.}
 \keywords{Event-triggered control, infinite-dimensional systems, Lipschitz perturbations, self-triggered control, stability.}

 \email{wakaiki@ruby.kobe-u.ac.jp}
 \email{sano@crystal.kobe-u.ac.jp}

\thanks{The first author is supported by JSPS KAKENHI 
	Grant Numbers JP20K14362}

\thanks{$^*$ Corresponding author: wakaiki@ruby.kobe-u.ac.jp}

\newcommand{\dom}{\mathop{\rm dom}\nolimits}

\begin{document}
\maketitle

\centerline{\scshape Masashi Wakaiki$^*$}
\medskip
{\footnotesize
 \centerline{Graduate School of System Informatics, Kobe University}
   \centerline{Nada, Kobe, Hyogo 657-8501, Japan}
} 

\medskip

\centerline{\scshape Hideki Sano}
\medskip
{\footnotesize
 \centerline{Graduate School of System Informatics, Kobe University}
\centerline{Nada, Kobe, Hyogo 657-8501, Japan}
}

\bigskip

 \centerline{(Communicated by the associate editor name)}

\begin{abstract}
This paper addresses the following question: ``Suppose 
that  a state-feedback controller stabilizes an infinite-dimensional linear continuous-time
system. If we choose the parameters of 
an event/self-triggering mechanism appropriately, is
the event/self-triggered control system 
stable under all sufficiently small nonlinear Lipschitz perturbations?''
We assume that the stabilizing feedback operator is compact.
This assumption is used to guarantee the strict positiveness of inter-event times
and the existence  of the mild solution of evolution equations
with unbounded control operators.
First, for  the case where the control operator is bounded,
we show that 
the answer to the above question is positive, giving
a sufficient condition  for exponential stability, which can be
employed for the design of event/self-triggering mechanisms.
Next, we investigate the case where the control operator is unbounded and prove that
the answer is still positive for periodic event-triggering mechanisms.
\end{abstract}

\section{Introduction}
In this paper, we study event/self-triggered control for infinite-dimensional systems.
As the time-discretization of control systems,
periodic sampling and control-updating are widely used.
Various problems on periodic sampled-data control have been studied for infinite-dimensional systems;
for example, stabilization \cite{Tarn1988, Rebarber1998, Logemann2005, Logemann2013, Karafyllis2018,Kang2018Automatica,Wakaiki2020SCL,
Lin2020}, 
robustness  analysis of continuous-time stabilization with respect to
periodic sampling \cite{Rebarber2002, Logemann2003, Rebarber2006}, and output regulation 
\cite{Logemann1997, Ke2009SCL, Ke2009SIAM, Ke2009IEEE, Wakaiki2019}.
Event/self-triggering mechanisms are other time-discretization methods, which
send measurements and update control inputs when they are needed.
In event-triggered control systems, a sensor monitors the plant output and 
determines when it sends the data to a controller (Figure~\ref{fig:ETC}).
On the other hand, in self-triggered control systems,
the controller computes times at which
the sensor transmits the data to the controller (Figure~\ref{fig:STC}).
The major difference is that the current output can be used to determine 
transmission times in event-triggered control systems but
not in self-triggered control systems.
Therefore, in the state-feedback case where we denote by $x(t)$ 
the state  at time $t\geq 0$,
the transmission times $\{t_k\}_{k \in \mathbb{N}_0}$ 
of event-triggered control systems 
are typically computed 
in the form $t_{k+1} = \inf \big\{ t>t_k: F_{\rm e}\big(x(t),x(t_k)\big) >0 \big\}$
and those of self-triggered control systems are in the form 
$t_{k+1} = t_k + F_{\rm s}\big(x(t_k) \big)$
for some functions $F_{\rm e}$ and $F_{\rm s}$.

\begin{figure}[tb]
	\centering
	\includegraphics[width = 6.5cm]{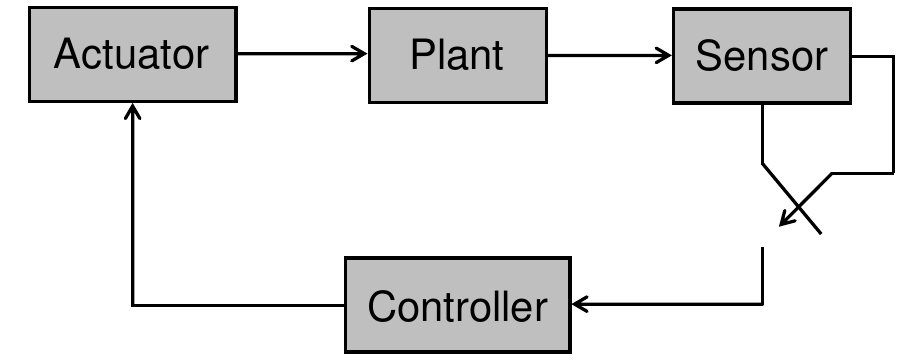}
	\caption{Event-triggered control system.}
	\label{fig:ETC}
\end{figure}

\begin{figure}[tb]
	\centering
	\includegraphics[width = 6.5cm]{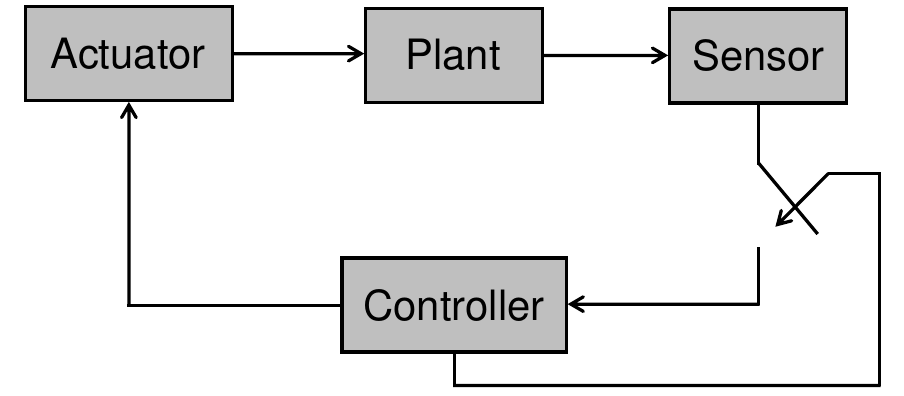}
	\caption{Self-triggered control system.}
	\label{fig:STC}
\end{figure}

Event/self-triggered control has been an area of intense research,
starting from the seminal works of
Tabuada \cite{Tabuada2007}, Wang and Lemmon \cite{Wang2009}, Anta and Tabuada~\cite{Anta2010}  for finite-dimensional systems.
Adaptation of sampling periods in sampled-data systems is a topic close to event/self-triggered control, which
has been also explored in \cite{Ilchmann2011} for finite-dimensional systems.
Event-triggered control methods have been extended to some classes of infinite-dimensional systems,
e.g., systems with output delays and packet losses \cite{Lehmann2012}, 
first-order hyperbolic systems \cite{Espitia2016,Espitia2018},
second-order hyperbolic systems \cite{Baudouin2019}, 
second-order parabolic systems \cite{Jiang2016, Selivanov2016PDE,Karafyllis2019},
and abstract linear evolution equations \cite{Wakaiki2018_EVC}. 
However, relatively little work has been done on 
self-triggered control for infinite-dimensional systems,
compared with event-triggered control.

We 
consider the following system with state space $X$ and input space $U$
(both Hilbert spaces):
\begin{equation}
\label{eq:state_eq_intro}
\dot x(t) = Ax(t)+Bu(t) + \phi\big(
x(t)
\big), \quad t \geq 0;\qquad x(0) = x^0 \in X,
\end{equation}
where $A$ is the generator of a strongly continuous semigroup $T(t)$ on $X$,
the control operator $B$ is a bounded linear operator from $U$ to the extrapolation space $X_{-1}$
associated with $T(t)$ (see the notation section for the definition of the extrapolation space $X_{-1}$), and 
the perturbation $\phi$ is a nonlinear operator on $X$ satisfying $\phi(0) = 0$ and the 
Lipschitz condition
\[
\|\phi(\xi_1) -\phi(\xi_2) \| \leq L \|\xi_1 - \xi_2\|\qquad \forall \xi_1,\xi_2 \in X
\]
for some Lipschitz constant $L \geq 0$. 
We call the control operator $B$ in \eqref{eq:state_eq_intro} 
{\em bounded} if
$B$ maps boundedly from $U$ into $X$.
Otherwise, we call $B$ {\em unbounded}.

Choose  a bounded linear operator $F$ from $X$ to $U$ such that 
the state-feedback controller $u(t) = Fx(t)$ exponentially stabilizes the linear system
$\dot x(t) = Ax(t) +Bu(t)$, that is, 
the strongly continuous semigroup generated by
$A+BF$ is exponentially stable.
For the infinite-dimensional system \eqref{eq:state_eq_intro},
we here implement an event/self-triggering controller, which is given by
\begin{equation}
\label{eq:control_eq_intro}
u(t) = Fx(t_k),\quad t_k \leq t < t_{k+1},~k \in \mathbb{N}_0,
\end{equation}
where $\{t_k\}_{k\in \mathbb{N}_0}$ is determined by a certain
event/self-triggering mechanism.
If we appropriately choose the parameters of the event/self-triggering mechanism,
then the inter-event times $t_{k+1}-t_k$ can be small. Therefore, we would expect
intuitively that the event/self-triggered controller \eqref{eq:control_eq_intro} exponentially
stabilizes the system \eqref{eq:state_eq_intro} for all perturbations
with sufficiently small Lipschitz constants $L$.
The main objective of this paper is to show that 
this intuition is correct.

In addition to stabilization,  the minimum inter-event time, $\inf_{k \in \mathbb{N}_0} (t_{k+1} - t_k)$,
needs to be strictly positive. Otherwise, data transmissions might occur infinitely fast, which
cannot be executed in practical implementation.
This is an example of Zeno behavior studied in the context of hybrid systems; see, e.g., \cite{Goebel2009}. 
In the infinite-dimensional case,
a careful treatment of the minimum inter-event time is required even for
state-feedback control. 
In fact, for finite-dimensional linear systems, if we use the following standard event-triggering mechanism
proposed in \cite{Tabuada2007}:
\begin{equation}
\label{eq:standard_ETM}
t_{k+1} = \inf \{ t>t_k:
\|x(t_k) - x(t)\| > \varepsilon \|x(t)\|
\},\quad k \in \mathbb{N}_0,~\varepsilon >0,
\end{equation}
then for every threshold $\varepsilon >0$, there exists $\theta>0$ such that 
$\inf_{k \in \mathbb{N}_0} (t_{k+1} - t_k) \geq \theta$ 
for all $x^0 \in X$; see Corollary~IV.1 of \cite{Tabuada2007}.
However, for infinite-dimensional linear systems,
the same mechanism may not guarantee that 
the minimum inter-event time is bounded from below by a positive constant
as shown in  Examples~3.1 and 3.2 in \cite{Wakaiki2018_EVC}.

The important assumption of this paper is that 
the feedback operator $F$ is compact, which
is used for two purposes.
First, we guarantee the strict positiveness of inter-event times,
using the compactness of the feedback operator.
Second, this assumption is employed to prove the existence of
the mild solution of the 
evolution equation \eqref{eq:state_eq_intro} and \eqref{eq:control_eq_intro}
in the unbounded control case.

We start with the case in which $B$ is bounded.
In the previous work \cite{Wakaiki2018_EVC}, the following
event-triggering mechanism has been considered for
the system \eqref{eq:state_eq_intro} in the unperturbed case $\phi \equiv 0$:
\begin{align}
t_{k+1} &= 
\min\big\{ t_k + \tau_{\max},~\inf \{ t> t_k:
\|Fx(t_k) - Fx(t)\|_U > \varepsilon \|x(t_k)\|
\}
\big\} \label{eq:ETM1_intro}
\end{align}
for $k \in \mathbb{N}_0$, where
$\tau_{\max}>0$ is
an upper bound of  inter-event times.
The self-triggering 
mechanism we propose in this paper is based on
\eqref{eq:ETM1_intro}. Before describing it, we briefly explain
two different points of 
the event-triggering mechanism \eqref{eq:ETM1_intro}
from the standard one \eqref{eq:standard_ETM}.

First, the mechanism \eqref{eq:ETM1_intro} has the upper bound $\tau_{\max}$ of inter-event times.
Since $\|x(t_k)\|$ is used for the estimation of the 
implementation-induced error
$\|Fx(t_k) - Fx(t)\|_U$
in the mechanism \eqref{eq:ETM1_intro}, 
exponential convergence is not guaranteed unless we set
an upper bound of inter-event times.
The
exponential stability of the unperturbed event-triggered control system 
is achieved for every $\tau_{\max}>0$ under
some condition on the threshold $\varepsilon$, 
although the decay rate becomes small as
$\tau_{\max}$ increases; see Theorem~4.1 in \cite{Wakaiki2018_EVC}.

Second, the implementation-induced error of the input, $\|Fx(t_k) - Fx(t)\|_U$, 
is used in \eqref{eq:ETM1_intro}.
As in the case of the mechanism \eqref{eq:standard_ETM},
there exists an infinite-dimensional systems such that
an event-triggering mechanism using the condition
$\|x(t_k) - x(t)\| > \varepsilon \|x(t_k)\|$ does
not guarantees that the minimum inter-event time 
is bounded from below by positive constant;
see Example~3.1 in \cite{Wakaiki2018_EVC}. However,
it has been shown in Theorem~3.6 of \cite{Wakaiki2018_EVC} that 
the minimum inter-event time of the 
mechanism \eqref{eq:ETM1_intro} is
strictly positive if the feedback operator $F$ is compact.

Event-triggering mechanisms using the feedback 
operator $F$ are not practical in some situations.
This is the main motivation of the extension of \eqref{eq:ETM1_intro}
to a self-triggering mechanism.
It is reasonable that the controller uses
the mechanism \eqref{eq:ETM1_intro} for the computation
of transmission times at which
the controller sends the control input to the actuator.
However, it makes little sense that 
the sensor uses the mechanism \eqref{eq:ETM1_intro} in the situation where
the sensor and the controller are separated.
Indeed, since the control input is computed in the mechanism \eqref{eq:ETM1_intro},
the sensor may directly transmit the control input to the actuator without going through the controller;
see also the discussion in Section~VII-B of \cite{Tabuada2007}.

For the bounded control case, we propose the following self-triggering mechanism:
\begin{align}
t_{k+1} &= t_k + 
	\min\big \{ \tau_{\max},~ \inf \{ \tau>0:
	\alpha_{L,\varepsilon}(x(t_k),\tau) \geq \varepsilon \|x(t_k)\| \}
	\big\}
\label{eq:STM_intro}
\end{align}
for $k \in \mathbb{N}_0$,
where $\alpha_{L,\varepsilon}:X \times \mathbb{R}_+ \to \mathbb{R}_+$ is a certain function depending on $L$ and $\varepsilon$;
see Section~\ref{sec:STM} for details.
Instead of monitoring the state $x(t)$ continuously,
the self-triggering mechanism \eqref{eq:STM_intro} predicts it 
to estimate $\|Fx(t_k) - Fx(t)\|_U$,
by using the nominal linear model $(T(t),B,F)$, 
the Lipschitz constant $L$, and
the latest transmitted state $x(t_k)$. Consequently,
the self-triggering mechanism \eqref{eq:STM_intro} can be
implemented in the controller.
The function $\alpha_{L,\varepsilon}$ is defined so that 
\[
\|Fx(t_k) - Fx(t)\|_U \leq 
\alpha_{L,\varepsilon}(x(t_k),t-t_k)
\]
holds for every $t \in [t_k, t_{k+1})$ and $k \in \mathbb{N}_0$.
Therefore, under the self-triggering mechanism \eqref{eq:STM_intro}, 
we obtain
\[
\|Fx(t_k) - Fx(t)\|_U  \leq \varepsilon \|x(t_k)\|
\]
for every $t \in [t_k, t_{k+1})$ and $k \in \mathbb{N}_0$
as under the event-triggering mechanism \eqref{eq:ETM1_intro}.
We show the strict positiveness of  inter-event times 
and provide a sufficient condition for
the exponential stability of the self-triggered control system.

Another solution to avoid the use of the feedback operator $F$ in the sensor
is the following event-triggering mechanisms:
\begin{equation}
t_{k+1} = 
\min\big\{ t_k + \tau_{\max},~\inf \{ t> t_k+ \tau_{\min}:
\|x(t_k) - x(t)\| > \varepsilon \|x(t_k)\|
\}
\big\}
\label{eq:ETM2_intro}
\end{equation}
for $k \in \mathbb{N}_0$, where $\tau_{\min} \in (0,\tau_{\max})$ is
a prespecified lower bound of  inter-event times.
The event-triggering mechanism  \eqref{eq:ETM2_intro}
is based on the one studied in \cite{Heemels2008} for
finite-dimensional systems.
The difficulty here is that 
the  event-triggering mechanism  \eqref{eq:ETM2_intro} does not 
guarantee that the error $\|x(t_k)-x(t_k +\tau)\|$ is small for $0<\tau < \tau_{\min}$.
However, we show that if  the lower bound $\tau_{\min}$ is chosen appropriately, then
exponential stability is preserved under 
the  event-triggering mechanism  \eqref{eq:ETM2_intro} 
for all sufficiently small Lipschitz constants $L>0$ and 
thresholds $\varepsilon \geq 0$. 
For the unperturbed case $\phi \equiv 0$, we also provide a simple sufficient condition for exponential stability,
in which the upper bound $\tau_{\max}$ of inter-event times 
does not appear as in
the case of the event-triggering mechanism \eqref{eq:ETM1_intro}.
It is worthwhile noticing that the stability analysis under the
event-triggering mechanism \eqref{eq:ETM2_intro} is new even without Lipschitz perturbations.

Next we investigate the case in which $B$ is unbounded.
As in the unperturbed case $\phi \equiv 0$ in Section~5.2 of \cite{Wakaiki2018_EVC},
we apply the following periodic  event-triggering mechanism: 
\begin{equation}
	\label{eq:ETM3_intro}
	t_{k+1} := 
	\min\big\{ t_k + \ell_{\max}h,~\min \{ \ell h > t_k:
	\|x(t_k) - x(\ell h)\| > \varepsilon \|x(t_k)\|,~\ell \in \mathbb{N}
	\}
	\big\},
\end{equation}
where $h>0$ is a sampling period and
$\ell_{\max} \in \mathbb{N}$ determines an upper bound of
inter-event times.
The periodic event-triggering mechanism has been studied for finite-dimensional linear  systems
in \cite{Heemels2013} and then has been extended to finite-dimensional  nonlinear Lipschitz 
systems in \cite{Etienne2017}.
Compared with the above mechanisms \eqref{eq:standard_ETM}, \eqref{eq:ETM1_intro}, and \eqref{eq:ETM2_intro},
the periodic event-triggering mechanism \eqref{eq:ETM3_intro} behaves in a more time-triggering fashion, because
the condition is verified only periodically. This periodic aspect leads to 
several benefits. First, the minimum inter-event time is bounded from below 
by $h$. Second, the sensor needs to monitor the state and check the condition
only at sampling times, and hence the periodic event-triggering mechanism \eqref{eq:ETM3_intro}
is more suitable for practical implementations. 

In the unbounded control case,
we begin by showing that a mild solution of the evolution equation \eqref{eq:state_eq_intro} and \eqref{eq:control_eq_intro}
uniquely exists in $C(\mathbb{R}_+,X)$. The difficulty arises from
the combination of the unboundedness of the control operator $B$ and 
the nonlinearity of the perturbation $\phi$. To solve the difficulty, 
we use the fact that if $F$ is compact, then $S_hF x^0$ is continuous with respect to $h$ 
in the norm of $X$ for every $x^0 \in X$, where 
\[
S_h :=\int^h_0 T(s)Bds.
\]
Next, assuming that the operator 
$
\Delta_h := T(h)+ S_hF,
$
which governs
the state evolution of the discretized system, is power stable, 
we extend the stability analysis developed in 
Section~5.2 of \cite{Wakaiki2018_EVC}
to the perturbed case and show that the
exponential stability of
the periodic event-triggered control system is achieved
for all sufficiently small Lipschitz constants $L>0$ and 
thresholds $\varepsilon \geq 0$. 
This is only an existence result, because
it is generally difficult in the unbounded control case to estimate how small
the sampling period $h$ has to be in order to achieve 
the exponential stability of the periodic sampled-data system.
However,  returning to the bounded control case, we
develop a simple sufficient condition  for exponential stability.
Similarly to the mechanisms \eqref{eq:ETM1_intro} and \eqref{eq:ETM2_intro},
the upper bound $\ell_{\max}h$ of inter-event times does not appear in
this sufficient condition,
and the decay rate of the closed-loop system 
becomes small  as $\ell_{\max}h$ increases.

This paper  is organized as follows.
In Section~\ref{sec:bounded_control}, we consider the case in which
the control operator $B$ is bounded. First we 
analyze the minimum inter-event time and the exponential stability
of the self-triggered control system with the mechanism \eqref{eq:STM_intro}.
Next, exponential stability under
the event-triggering mechanism \eqref{eq:ETM2_intro} is studied.
Before proceeding to the unbounded control case,
we provide a numerical example  in Section~\ref{sec:numerical_example}
to illustrate the proposed event/self-triggering mechanisms for 
bounded control case, where
a heat equation in cascade with an ordinary differential equation (ODE) is considered.
In Section~\ref{sec:PET}, we study the case in which
the control operator $B$ is unbounded.
After proving that a mild solution of the evolution equation uniquely exists,
we apply the periodic event-triggering
mechanism \eqref{eq:ETM3_intro} 
to infinite-dimensional systems with Lipschitz perturbations.

\subsection*{Notation}
We denote by $\mathbb{Z}$ and $\mathbb{N}$ the set of  integers and
the set of positive integers, respectively.
Define 
$\mathbb{N}_0 := \mathbb{N} \cup \{0\}$ and $\mathbb{R}_+ := [0,\infty)$.
Let $X$ and $Y$ be Banach spaces. Let us denote 
the space of all bounded linear operators from $X$ to $Y$ 
by $\mathcal{B}(X,Y)$, and set $\mathcal{B}(X) := \mathcal{B}(X,X)$.
Denote by $\mathcal{K}(X,Y)$ the closed subspace of $\mathcal{B}(X,Y)$ consisting
of all compact operators.
Let $A$ be a linear operator from $X$ to $Y$. 
The domain of $A$ is denoted by $\dom (A)$.
The resolvent set  of a linear operator $A:\dom (A) \subset X \to X$
is denoted by $\varrho(A)$.
For an interval $I \subset \mathbb{R}$, 
we denote by $C(I,X)$ the space of all continuous functions $f:I \to X$ and
by $C^1(I,X)$
the space of all continuously differentiable functions $f:I \to X$.
For $p \geq 1$, we denote by $L^p(I,X)$ the space of
all measurable functions $f:I \to X$ such that $\int_I \|f(t)\|^pdt < \infty$.

Let $X$ be a Banach space.
An operator $\Delta \in \mathcal{B}(X)$ is said to be {\em power stable} if there exist
$\Omega \geq  1$ and $\omega \in (0,1)$ such that 
$\|\Delta^k\|_{\mathcal{B}(X)}  \leq \Omega \omega^k$ for every $k \in \mathbb{N}_0$.
Let  $T(t)$ be
a strongly continuous semigroup on $X$.
We say that $T(t)$ is {\em exponentially stable}
if  
there exist $\Gamma \geq 1$ and $\gamma >0$ satisfy
$\|T(t)\|_{\mathcal{B}(X)} \leq \Gamma e^{-\gamma t}$ for all $t \geq 0$.
Let $A$ be the generator of $T(t)$.
For $\lambda \in \varrho(A)$,
the extrapolation space $X_{-1}$ associated with $T(t)$ is the completion of $X$ with
respect to the norm $\|x \|_{-1} := \|(\lambda I - A)^{-1}x \|$.
Different choices of $\lambda$ lead to equivalent norms on $X_{-1}$.
The semigroup $T(t)$ can be extended to a strongly continuous semigroup 
on $X_{-1}$, and its generator on $X_{-1}$ is an extension of $A$ to $X$.
We shall use the same symbols $T(t)$ and $A$
for the original ones and the associated extensions.
We refer the reader to Section~II.5 in \cite{Engel2000} and
Section~2.10 in \cite{Tucsnak2009} for more details
on the extrapolation space $X_{-1}$.

\section{Event/self-triggering mechanisms for bounded control}
\label{sec:bounded_control}
In this section, 
we study event/self-triggered control systems with
bounded control operators, i.e, $B \in \mathcal{B}(U,X)$.
First, we introduce the infinite-dimensional system
considered here.
Next, we propose a self-triggering mechanism
employing the input error and then 
analyze the minimum inter-event time
and the exponential stability of the self-triggered control system.
Finally, we study the exponential stability 
of event-triggered control systems with mechanisms
in which a lower bound of the minimum inter-event time is prespecified.
\subsection{Plant dynamics and preliminaries}
Let us denote by $X$ and $U$ the state space and the input space, 
and both of them are Hilbert spaces.
We denote by $\|\cdot\|$  the norm of $X$.
Take  $\tau_{\min} >0$, and
let an increasing sequence $\{t_k\}_{k\in\mathbb{N}_0}$ satisfy $t_0 = 0$ and 
$t_{k+1}-t_k \geq \tau_{\min}$ for every $k \in \mathbb{N}_0$.
Consider the following infinite-dimensional system:
\begin{subequations}
	\label{eq:plant}
	\noeqref{eq:state_equation,eq:input}
	\begin{align}
	\label{eq:state_equation}
	\dot x (t) &= Ax(t) +Bu(t) + \phi\big(x(t)\big),\quad t \geq 0; \qquad 
	x(0) = x^0 \in X \\
	\label{eq:input}
	u(t) &= F x(t_k),\quad  t_k \leq t < t_{k+1},~k \in \mathbb{N}_0,
	\end{align}
\end{subequations}
where $x(t) \in X$ is the state and $u(t) \in U$ is the input for $t\geq 0$.
We assume that 
$A$ is the generator of a strongly continuous semigroup $T(t)$ on $X$.
The control 
operator $B$ and the feedback operator $F$ satisfy
$B \in \mathcal{B}(U,X)$ and $F \in  \mathcal{B}(X,U)$, respectively.
The perturbation $\phi:X\to X$ is a nonlinear operator satisfying 
the Lipschitz  condition
\begin{equation}
\label{eq:Lip_cond}
\phi(0) = 0,\quad \|\phi(\xi_1) -\phi(\xi_2) \| \leq L \|\xi_1 - \xi_2\|\qquad \forall \xi_1,\xi_2 \in X
\end{equation}
for some Lipschitz constant $L \geq 0$.

To study the solution of 
the evolution equation \eqref{eq:plant}, we consider
the following integral equation:
\begin{subequations}
	\label{eq:mild_solution}
	\noeqref{eq:mild_solution_initial,eq:mild_solution_eq}
	\begin{align}
	x(0) &= x^0 \in X \label{eq:mild_solution_initial}\\
	x(t_k+\tau) &= T(\tau) x(t_k) + \int^\tau_0
	T(\tau - s) 
	\Big(
	BF x(t_k) + \phi\big(x(t_k+s) \big)
	\Big) ds\label{eq:mild_solution_eq} 
	\end{align}
\end{subequations}
for all $\tau \in (0,t_{k+1}-t_k]$ and all $k \in \mathbb{N}_0$.
The
integral equation \eqref{eq:mild_solution} has a unique solution in $C(\mathbb{R}_+,X)$ by
Theorem~1.2 on p.~184 of \cite{Pazy1983}.
Moreover, this solution satisfies the evolution equation \eqref{eq:plant} in a certain
sense; see, e.g., Theorem~\ref{prop:existence_solution} 
in the unbounded control case for
details.
We say that the continuous solution of the
integral equation \eqref{eq:mild_solution} is
a (mild) solution of the evolution equation \eqref{eq:plant}.

We define the exponential stability of the closed-loop system \eqref{eq:plant}.
\begin{definition}[Exponential stability]
	The closed-loop system \eqref{eq:plant} is {\em exponentially stable} if
	there exist $\Gamma \geq 1$ and $\gamma >0$ such that 
	the solution $x$ of the integral equation \eqref{eq:mild_solution} satisfies
	\[
	\|x(t)\| \leq \Gamma e^{-\gamma t} \|x^0\|\qquad \forall x^0 \in X,~\forall t\geq0.
	\]
\end{definition}

Define the operators $S_\tau \in \mathcal{B}(U,X)$ and
$\Delta_\tau \in \mathcal{B}(X)$ by
\begin{equation}
\label{eq:Delta_def}
S_\tau := \int^{\tau}_0 T(s) Bds,\qquad \Delta_\tau := T(\tau) + S_\tau F.
\end{equation}
Using 
this operator $\Delta_\tau$, 
we can rewrite \eqref{eq:mild_solution_eq} as
\begin{equation}
\label{eq:x_Delta}
x(t_k + \tau) = \Delta_\tau x(t_k) + \int^\tau_0 T(\tau-s) \phi \big(
x(t_k+s)
\big) ds.
\end{equation}
for all $\tau \in (0,t_{k+1}-t_k]$ and all $k \in \mathbb{N}_0$.

A simple calculation (see, e.g., Exercise 3.3 on p.~129 in \cite{Curtain1995}) yields the following
equivalence of solutions:
\begin{lemma}
	\label{lem:another_rep_MS}
	Let $\tau >0$. 
	Assume that $A$ generates a strongly continuous semigroup $T(t)$ on $X$, 
	$B \in \mathcal{B}(U,X)$, $F \in \mathcal{B}(X,U)$, and $u \in L^1([0,\tau),U)$.
	Then the mild solution of
	\[
	\dot x(t) = Ax(t)+Bu(t),\quad 0\leq t < \tau ;\qquad x(0) = x^0 \in X \\
	\]
	equals the mild solution of 
	\[
	\dot x(t) = (A+BF)x(t)+B[u(t) - Fx(t)],\quad 0\leq t < \tau;\qquad x(0) = x^0 \in X.
	\]	
\end{lemma}

Let $T_{BF}(t)$ denote the strongly continuous semigroup generated by $A+BF$. 
Since
the evolution equation \eqref{eq:plant} is rewritten as
\[
\dot x(t) = (A+BF)x(t) + \phi\big(x(t) \big) + BF [x(t_k) - x(t)],\quad t_k \leq t < t_{k+1},~k \in \mathbb{N}_0,
\]
Lemma~\ref{lem:another_rep_MS}
yields another representation of
the solution $x$ given in \eqref{eq:mild_solution_eq}:
\begin{align}
x(t_k+\tau) = T_{BF}(\tau)x(t_k) 
&+\int^\tau_0 T_{BF}(\tau-s) \phi\big(
x(t_k+s)
\big)ds \notag \\
&+
\int^\tau_0 T_{BF}(\tau-s) BF 
[
x(t_k) - x(t_k+s)
]ds 
\label{eq:x_TBF_rep}
\end{align}
for every $\tau \in (0,t_{k+1}-t_k]$ and every $k \in \mathbb{N}_0$.

The following lemma provides an upper bound of the state norm
between transmission times.
\begin{lemma}
	\label{lem:inter_sample_bounded_case}
	Assume that 
	the semigroup $T_{BF}(t)$ generated by $A+BF$ is exponentially stable,
	i.e, 
	\begin{equation}
	\label{eq:TBF_exp_stable}
	\|T_{BF}(t)\|_{\mathcal{B}(X)} \leq \Gamma e^{-\gamma t}\qquad \forall t \geq 0.
	\end{equation} 
	holds for some $\Gamma \geq 1$ 
	and $\gamma>0$.
	If the solution $x$ of the integral equation~\eqref{eq:mild_solution} 
	satisfies
	\begin{equation}
	\label{eq:Fx_diff_bound}
	\|Fx(t_k) - Fx(t_k+\tau)\|_U \leq \varepsilon \|x(t_k)\|\qquad
	\forall \tau\in [0, t_{k+1}-t_k),~ \forall k \in \mathbb{N}_0
	\end{equation}
	for some $\varepsilon >0$, then 
	\begin{equation}
	\label{eq:x_bound_nonlinear}
	\|x(t_{k}+\tau)\| \leq \Gamma e^{\Gamma L\tau} [(1-b_1\varepsilon) e^{-\gamma \tau} + b_1\varepsilon]
	\|x(t_k)\|\quad 
	\forall \tau \in [0, t_{k+1}-t_k),~ \forall k \in \mathbb{N}_0,
	\end{equation}
	where $b_1:= \|B\|_{\mathcal{B}(U,X)} /\gamma $.
\end{lemma}
\begin{proof}
	Let $k \in \mathbb{N}_0$ and $\tau \in [0, t_{k+1}-t_k)$.
	Using 
	\eqref{eq:x_TBF_rep} and
	\eqref{eq:Fx_diff_bound},
	we obtain
	\begin{align*}
		\|x(t_k+\tau)\| &\leq \Gamma e^{-\gamma \tau} \|x(t_k)\| + 
		\Gamma\int^\tau_0  e^{-\gamma(\tau - s)} \big\|\phi\big(x(t_k+s) \big)\big\|ds \\
		&\qquad  + 
		\Gamma\|B \|_{\mathcal{B}(U,X)}  \int^\tau_0 e^{-\gamma(\tau - s)}\| 
		F
		x(t_k) - Fx(t_k+s)
		\|_Uds\\
		&=
		\Gamma [(1-b_1\varepsilon) e^{-\gamma \tau} + b_1\varepsilon] \|x(t_k)\| + 
		\Gamma L \int^\tau_0 e^{-\gamma (\tau -s)} \|x(t_k+s)\| ds,
	\end{align*}
	where $b_1:= \|B\|_{\mathcal{B}(U,X)} /\gamma $.
	Define $y (t_k+\tau) := e^{\gamma \tau} \|x(t_k+\tau)\|$. Then
	\[
	y(t_k+\tau) \leq \Gamma  [(1-b_1\varepsilon) + b_1\varepsilon e^{\gamma \tau}]
	y(t_k)
	+\Gamma L \int^\tau_0 y(t_k+s) ds.
	\]
	Gronwall's inequality yields
	\[
	y(t_k+\tau) \leq \Gamma e^{\Gamma L\tau} [(1-b_1\varepsilon) + b_1\varepsilon e^{\gamma \tau}]
	y(t_k).
	\]
	Thus, we obtain the desired estimate \eqref{eq:x_bound_nonlinear}.
\end{proof}

We write the coefficient of $\|x(t_k)\|$ in \eqref{eq:x_bound_nonlinear}
as $\eta_{L,\varepsilon}(\tau)$:
\begin{equation}
\label{eq:chi_def}
\eta_{L,\varepsilon}(\tau) := \Gamma e^{\Gamma L\tau} [(1-\varepsilon\|B\|_{\mathcal{B}(U,X)} /\gamma ) 
e^{-\gamma \tau} + \varepsilon\|B\|_{\mathcal{B}(U,X)} /\gamma ],\quad \tau \geq 0.
\end{equation}

Suppose that  the semigroup $T_{BF}(t)$ is exponentially stable,
i.e., \eqref{eq:TBF_exp_stable} holds for some $\Gamma \geq 1$ and $\gamma >0$. Then
we define a new norm $|\cdot|$ on $X$ by
\begin{align}
	\label{eq:new_norm_cont}
	|x| := \sup_{t\geq 0} \|e^{\gamma t} T_{BF}(t) x\|,\quad x \in X
\end{align}
as in the proof of
Theorem 3.1 in \cite{Logemann2003}.
It has been shown there that this norm satisfies
\begin{equation}
\label{eq:new_norm_cont_prop}
\|x\| \leq |x| \leq \Gamma \|x\|,\quad |T_{BF}(t) x| \leq e^{-\gamma t}  |x|\qquad \forall x \in X,~\forall t \geq 0.
\end{equation}

\subsection{Self-triggering mechanism employing input errors}
\label{sec:STM}

Throughout this and next subsections, we place the following assumption.
\begin{assumption}
	\label{assump:bounded_case}
		Assume that $A$ generates a strongly continuous semigroup $T(t)$ on $X$, 
	$B \in \mathcal{B}(U,X)$, $F \in  \mathcal{K}(X,U)$, 
	and a nonlinear operator $\phi:X\to X$ satisfies 
	the Lipschitz condition \eqref{eq:Lip_cond}.
	Moreover, assume that 
	the semigroup $T_{BF}(t)$ generated by $A+BF$ is exponentially stable,
	i.e, \eqref{eq:TBF_exp_stable} holds for some $\Gamma \geq 1$ 
	and $\gamma>0$.
\end{assumption}

We propose a self-triggering mechanism that constructs 
$\{t_k\}_{k \in \mathbb{N}_0}$ only from the data on 
the nominal linear model $(T(t),B,F)$, 
the Lipschitz constant $L$, 
and the latest transmitted state $x(t_k)$.
For $\xi \in X$ and $\tau \geq 0$, define
\[
\alpha_{L,\varepsilon}(\xi,\tau) := 
\|F(I-\Delta_{\tau}) \xi\|_U + 
L
\int_0^\tau \|FT(\tau - s)\|_{\mathcal{B}(X,U)}  \eta_{L,\varepsilon}(s) ds
\|\xi\|,
\]
where $\eta_{L,\varepsilon}$ is defined by \eqref{eq:chi_def}.
We consider the following self-triggering mechanism:
\begin{subequations}
	\label{eq:STC_Feedback}
	\noeqref{eq:STC_Feedback1,eq:STC_Feedback2}
	\begin{align}
	&t_{k+1} := t_k + 
	\min\{ \tau_{\max},~\tau_k
	\};\quad t_0 := 0  \label{eq:STC_Feedback1}\\
	&\tau_k := \inf \big\{ \tau>0:
	\alpha_{L,\varepsilon}(x(t_k),\tau) \geq \varepsilon \|x(t_k)\|
	\big\},\qquad k \in \mathbb{N}_0,\label{eq:STC_Feedback2}
	\end{align}
\end{subequations}
where $\varepsilon >0$ is a threshold parameter and 
$\tau_{\max}>0$ is
an upper bound of  inter-event times, i.e., $t_{k+1} - t_k \leq \tau_{\max}$
for every $k \in \mathbb{N}_0$.
The important feature of this mechanism is to determine 
the transmission times $\{t_k\}_{k \in \mathbb{N}_0}$
without using the present state $x(t)$.
Therefore, it can be implemented at the controller.

\begin{remark}[Role of $\tau_{\max}$]
By \eqref{eq:x_bound_nonlinear}, we only have
$\limsup_{\tau \to \infty} \|x(t_k+\tau)\| \leq 
	\Gamma b_1\varepsilon \|x(t_k)\|$
	even in the unperturbed case $\phi \equiv 0$.
To achieve exponential stability, 
we set an upper bound $\tau_{\max}$ of 
the inter-event times when we use
triggering mechanisms 
that compare 
the last-released data 
$\|x(t_k)\|$, not the present data $\|x(t)\|$, with
an implementation-induced error such as 
$Fx(t_k) - Fx(t_k+\tau)$.
In Theorem~4.2 of \cite{Wakaiki2018_EVC},
an event-triggering mechanism that compares the present data $\|x(t)\|$
with the implementation-induced error  was
investigated for infinite-dimensional systems, and 
a sufficient condition for exponential stability
was obtained with the help of the
classical Lyapunov equation.
In this theorem, however, the rather restrictive assumption that
a lower bound on the decay of $T(t)$ is  strictly positive is placed.
The recent developments of
Lyapunov functions for input-to-state stability 
(see, e.g., \cite{Mironchenko2020}) may
yield interesting results on event/self-triggered control for
infinite-dimensional systems, but we leave it for future work.
\end{remark}

To investigate
the minimum inter-event time,
we use
the following result, in which the compactness of
the feedback operator $F$ plays an important role.
\begin{lemma}[Lemma 3.5 in \cite{Wakaiki2018_EVC}]
	\label{lem:FD_conv}
	Let $T(t)$ be a strongly continuous semigroup on $X$,
	$B \in \mathcal{B}(U,X)$, and 
	$F \in  \mathcal{K}(X,U)$.
	Then the operator $\Delta_\tau \in \mathcal{B}(X)$ defined by
	\eqref{eq:Delta_def} satisfies
	\begin{equation}
	\label{eq:FD_conv}
	\lim_{\tau \downarrow 0}
	\| F(I - \Delta_\tau) \|_{\mathcal{B}(X,U)} = 0.
	\end{equation}
\end{lemma}

Using this lemma, we now show that 
the minimum inter-event time of the self-triggered control system is
bounded from below by a positive constant.
Moreover, we 
provides a sufficient condition for
exponential stability.
\begin{theorem}
	\label{thm:STC_FB_stability}
	Under Assumption~\ref{assump:bounded_case},
	the following two statements hold:
	\begin{enumerate}
		\renewcommand{\labelenumi}{\alph{enumi})}
		\item For every $L \geq 0$ and $\varepsilon,\tau_{\max} >0$, 
		there exists $\theta >0$ such that for every $x^0 \in X$, 
		the increasing sequence $\{t_{k}\}_{k\in \mathbb{N}_0}$ 
		defined by the self-triggering mechanism  \eqref{eq:STC_Feedback}
		satisfies
		$
		\inf_{k \in \mathbb{N}_0} (t_{k+1} - t_k) \geq \theta.
		$
		\item 
		The system \eqref{eq:plant}
		with the self-triggering mechanism \eqref{eq:STC_Feedback}
		is exponentially stable if
		$L \geq 0$ and $\varepsilon,\tau_{\max} >0$ satisfy
		\begin{equation}
		\label{eq:L_thr_cond_STC}
		\max \left\{ 
		\frac{\Gamma}{\gamma} L ,~
		\varpi(L,\varepsilon,\tau_{\max})
		\right\} +\frac{\Gamma \|B\|_{\mathcal{B}(U,X)}}{\gamma}  \varepsilon  < 1,
		\end{equation}
		where 
		\begin{align}
		\varpi(L,\varepsilon,\tau) := 
		\frac{1}{e^{\gamma \tau}-1} 
		\bigg(
		&\left(1- \frac{\varepsilon\Gamma \|B\|_{\mathcal{B}(U,X)} }{\gamma }\right)\
		(e^{\Gamma L \tau} - 1)\notag \\
		&\quad \qquad + 
		\frac{\varepsilon\Gamma  \|B\|_{\mathcal{B}(U,X)} }{\gamma } \frac{\Gamma L(e^{(\Gamma L+\gamma) \tau} - 1)}{\Gamma L
			+\gamma}		
		\bigg).		\label{eq:theta_def}
		\end{align}
	\end{enumerate}
\end{theorem}
\begin{proof}
	a) 
	Let $L \geq 0$, and
	 $\varepsilon,\tau_{\max} >0$  be given.
	For the first term of $\alpha_{L,\varepsilon}$, we obtain
	\[
	\|F(I-\Delta_\tau)\xi\|_U  \leq \|F(I-\Delta_\tau)\|_{\mathcal{B}(X,U)}  ~\!\|\xi\|
	\qquad \forall \xi \in X,
	\]
	and 
	$
	\lim_{\tau \downarrow 0}\|F(I-\Delta_{\tau})\|_{\mathcal{B}(X,U)}  = 0
	$
	by Lemma~\ref{lem:FD_conv}.
	Moreover, the integral term of $\alpha_{L,\varepsilon}$,
	\[
	\int_0^\tau \|FT(\tau - s)\|_{\mathcal{B}(X,U)} ~\!
	\eta_{L,\varepsilon}(s) ds,
	\]
	is continuous with respect to $\tau$ (see, e.g., Proposition~1.3.2 on p.~22
	of \cite{Arendt2001} for 
	the continuity property of convolutions) and 
	goes to 0 as $\tau \to 0$.
	Hence $t_1 - t_0 \geq \theta$ for some $\theta >0$, and
	$\theta$ does not depend on the initial state $x^0$.
	Since
	$x(t_k) \in X$ for every $k \in \mathbb{N}$,
	we obtain $\inf_{k \in \mathbb{N}_0} (t_{k+1} - t_k)  \geq \theta$
	by induction.

	b)
	We first show that 
	\begin{equation}
	\label{eq:input_diff_STC_tk}
	\|Fx(t_k) - Fx(t_k+\tau)\|_U \leq \varepsilon \|x(t_k)\|\qquad \forall \tau \in [0,t_{k+1}-t_k),~\forall k \in \mathbb{N}_0.
	\end{equation}
	Assume, to get a contradiction, that 
	there exists $k \in \mathbb{N}_0$ such that 
	\[
	\tau_1 := \inf \{ \tau \geq 0:
	\|Fx(t_k) - Fx(t_k+\tau)\|_U >  \varepsilon \|x(t_k)\|
	\} \in [0,t_{k+1}-t_k).
	\]
	By the continuity of $x$, 
	\begin{equation}
	\label{eq:input_diff_STC}
	\|Fx(t_k) - Fx(t_k+\tau_1)\|_U =  \varepsilon \|x(t_k)\|.
	\end{equation}
	Moreover,
	\[
	\|Fx(t_k) - Fx(t_k+\tau)\|_U  \leq   \varepsilon \|x(t_k)\|\qquad
	\forall \tau \in [0,\tau_1],
	\]
	and hence \eqref{eq:x_Delta}
	and  Lemma~\ref{lem:inter_sample_bounded_case} 
	yield
	\begin{align*}
	\|Fx(t_k) - Fx(t_k+\tau)\|_U 
	&\leq
	\alpha_{L,\varepsilon} (x(t_k),\tau)\qquad 
	\forall \tau \in [0, \tau_1].
	\end{align*}
	Since $\tau_1 < t_{k+1} - t_k$, it follows from 
	the self-triggering mechanism \eqref{eq:STC_Feedback} 
	that 
	\[
	\alpha_{L,\varepsilon} (x(t_k),\tau) < \varepsilon \|x(t_k)\|\qquad 
	\forall \tau \in [0, \tau_1].
	\]
	This implies that 
	\[
	\|Fx(t_k) - Fx(t_k+\tau_1)\|_U <  \varepsilon \|x(t_k)\|,
	\]
	which contradicts \eqref{eq:input_diff_STC}.

	Using the same argument as in the proof of 
	Lemma~\ref{lem:inter_sample_bounded_case},
	we have from \eqref{eq:new_norm_cont_prop} and \eqref{eq:input_diff_STC_tk} that
	\[
	|x(t_k+\tau)| \leq e^{\Gamma L \tau} 
	[(1-b_{\rm s} \varepsilon) e^{-\gamma \tau} + b_{\rm s} \varepsilon] ~\!|x(t_k)|
	\]
	for every $\tau \in (0,t_{k+1}-t_k]$ and every $k \in \mathbb{N}_0$,
	where $b_{\rm s} := \Gamma \|B\|_{\mathcal{B}(U,X)} /\gamma $.
	Therefore,
	\begin{align*}
	\left|
	\int^\tau_0 T_{BF}(\tau-s) \phi\big(
	x(t_k+s)
	\big)ds
	\right| &\hspace{-0.6pt}\leq\hspace{-0.6pt}
	\Gamma L \int^{\tau}_0 e^{-\gamma(\tau-s)} |x(t_k+s)| ds  \\
	&\hspace{-0.6pt}\leq \hspace{-0.6pt}
	\Gamma L
	\int^{\tau}_0 e^{-\gamma(\tau-s)}  e^{\Gamma L s} [(1-b_{\rm s}) \varepsilon e^{-\gamma s} + b_{\rm s} \varepsilon] ds|x(t_k)|
	\end{align*}
	for every $\tau \in (0,t_{k+1}-t_k]$ and every $k \in \mathbb{N}_0$.
	Note that 
	$\varpi(L,\varepsilon,\tau)$ defined by
	\eqref{eq:theta_def} satisfies
	\[
	\varpi(L,\varepsilon,\tau)=
	\frac{\Gamma L}{1-e^{-\gamma \tau}} 
	\int^{\tau}_0 e^{-\gamma(\tau-s)}  e^{\Gamma L s} [(1-b_{\rm s} \varepsilon) e^{-\gamma s} + b_{\rm s} \varepsilon]   ds.
	\]
	Moreover, a routine calculation shows that 
	\begin{align*}
	\sup_{0< \tau \leq \tau_{\max}}
	\varpi(L,\varepsilon,\tau) &= 
	\max\left\{
	\lim_{\tau \downarrow 0} \varpi(L,\varepsilon,\tau), ~	\varpi(L,\varepsilon,\tau_{\max})
	\right\}\\
	&=
	\max \left\{ 
	\frac{\Gamma}{\gamma} L ,~
	\varpi(L,\varepsilon,\tau_{\max})
	\right\}  =: \varpi_{\rm s}(L,\varepsilon,\tau_{\max}).
	\end{align*}
	It follows that 
	\begin{equation}
	\label{eq:STC_nonlinear_bound}
	\left|
	\int^\tau_0 T_{BF}(\tau-s) \phi\big(
	x(t_k+s)
	\big)ds
	\right| \leq \varpi_{\rm s}(L,\varepsilon,\tau_{\max}) (1-e^{-\gamma \tau}) |x(t_k)|
	\end{equation}
	for every $\tau \in (0,t_{k+1}-t_k]$ and every $k \in \mathbb{N}_0$.
	
	Using \eqref{eq:new_norm_cont_prop} and \eqref{eq:input_diff_STC_tk} 
	again, we obtain
	\begin{equation}
	\label{eq:STM_error}
	\left|
	\int^\tau_0 T_{BF}(\tau-s) BF [
	x(t_k) - x(t_k+s)]ds
	\right|	\leq b_{\rm s} \varepsilon (1-e^{-\gamma \tau}) |x(t_k)|
	\end{equation}
	for every $\tau \in (0,t_{k+1}-t_k]$ and every $k \in \mathbb{N}_0$.
	Applying these 
	estimates \eqref{eq:STC_nonlinear_bound}
	and \eqref{eq:STM_error}
	to \eqref{eq:x_TBF_rep}, we have
	\[
	|x(t_k+\tau)| \leq e^{-\gamma \tau} |x(t_k)| + 
	(\varpi_{\rm s}(L,\varepsilon,\tau_{\max}) + b_{\rm s} \varepsilon) (1-e^{-\gamma \tau})  |x(t_k)|
	\]
	for every $\tau \in (0,t_{k+1}-t_k]$ and every $k \in \mathbb{N}_0$.
	By the condition \eqref{eq:L_thr_cond_STC},
	$
	\beta_{\rm s} (L,\varepsilon,\tau_{\max})  :=
	\varpi_{\rm s}(L,\varepsilon,\tau_{\max}) + b_{\rm s} \varepsilon < 1.
	$
	Define
\[
f_{\rm s}(\tau) := \frac{-\log \big(e^{-\gamma \tau} +
	\beta_{\rm s} (L,\varepsilon,\tau_{\max})  (1-e^{-\gamma \tau}) \big)}{\tau}.
\]
Since 
\begin{align*}
e^{-\gamma \tau} +
\beta_{\rm s} (L,\varepsilon,\tau_{\max})  (1-e^{-\gamma \tau}) &=
[1-\beta_{\rm s} (L,\varepsilon,\tau_{\max})]e^{-\gamma \tau} + \beta_{\rm s} (L,\varepsilon,\tau_{\max}) < 1
\end{align*}
for all $\tau >0$,
it follows that $f_{\rm s}(\tau) >0$ for every $\tau>0$.
Moreover, a straightforward calculation shows that 
$f_{\rm s}$ is monotonically decreasing on $(0,\infty)$.
Hence
\[
|x(t_k+\tau)| \leq e^{-\gamma_0 \tau} |x(t_k)|\qquad 
\forall \tau \in (0,t_{k+1}-t_k],~\forall k\in \mathbb{N}_0,
\]
where $\gamma_0 := f_{\rm s}(\tau_{\max}) > 0$.
By induction on $k\in \mathbb{N}_0$, we obtain
\[
\|x(t_k+\tau)\| \leq 
|x(t_k+\tau)| \leq 
e^{-\gamma_0 \tau} |x(t_k)| \leq
\Gamma e^{-\gamma_0 (t_k+\tau)} \|x^0\|
\]
for every $x^0 \in X$, $\tau \in (0,t_{k+1}-t_k]$, and
$k \in \mathbb{N}_0$.
Thus, the system \eqref{eq:plant}
with the self-triggering mechanism \eqref{eq:STC_Feedback}
is exponentially stable.
\end{proof}

\begin{remark}[Dependence on $\tau_{\max}$]
	The upper bound $\tau_{\max}$ of
	inter-event times affects the performance of the closed-loop system
	in the following two ways.
	First, 	$L$ and $\varepsilon$
	depend on $\tau_{\max}$ in \eqref{eq:L_thr_cond_STC}
	if $\varpi(L,\varepsilon,\tau_{\max}) \geq \Gamma L/ \gamma$,
	which occurs, e.g., for large $\tau_{\max}$ and $L$.
	Second, 
	the lower bound $\gamma_0$ of
	the decay rate  of the self-triggered control system 
	becomes smaller as $\tau_{\max}$ increases.
\end{remark}

\subsection{Event-triggering mechanism enforcing minimal inter-event time}
We define the increasing sequence $\{t_k\}_{k\in\mathbb{N}_0}$ by
\begin{subequations}
	\label{eq:ETC_LB}
	\noeqref{eq:ETM_LB1,eq:ETM_LB2}
	\begin{align}
	&t_{k+1} := 
	\min\{ t_k + \tau_{\max},~{\bar t}_{k+1}
	\} \label{eq:ETM_LB1};\quad t_0 := 0
	\\
	&{\bar t}_{k+1} := \inf \big\{ t> t_k + \tau_{\min}:
	\|x(t_k) - x(t)\| > \varepsilon \|x(t_k)\|
	\big\},\quad  k \in \mathbb{N}_0, \label{eq:ETM_LB2}
	\end{align}
\end{subequations}
where $\varepsilon \geq 0$
is a threshold parameter and 
 $\tau_{\max} > \tau_{\min} >0$ are upper and lower 
 bounds on inter-event times, respectively, i.e., 
 $\tau_{\min} \leq t_{k+1} - t_k \leq \tau_{\max}$
 for every $k \in \mathbb{N}_0$.
Here we consider the situation where 
high-performance sensors are used for the continuous measurement 
of the state but
the capacity of communication channels and the actuator capability 
are limited, i.e.,
we cannot transmit data or update control inputs so frequently.
Practically, $\tau_{\min}$ is first determined,
and then we choose the threshold $\varepsilon$ so that the event-triggered control system is exponentially stable.

If the threshold $\varepsilon = 0$, then
the event-triggered control system with the mechanism \eqref{eq:ETC_LB}
can be regarded as a periodic sampled-data system with 
sampling period $\tau_{\min}$.
Unlike the case of periodic sampling, 
the sensor needs to measure the state $x(t)$ continuously  after $t>t_{\min}$ in  
the event-triggering mechanism \eqref{eq:ETC_LB}.
Therefore, the processing load of 
the sensor is high in the event-triggered control system.
However, the event-triggering mechanism \eqref{eq:ETC_LB}
has a potential to 
reduce the number of data transmissions, because
it determines transmission times depending on the state $x(t)$.
We see it from numerical simulations in Section~\ref{sec:numerical_example}.

We first show that 
a suitable choice of the threshold $\varepsilon$ and
the lower bound $\tau_{\min}$ of inter-event times
makes the event-triggered control system exponentially stable
under all sufficiently small Lipschitz perturbations.
\begin{theorem}
	\label{thm:ETC_LB_stability}
	Suppose that Assumption~\ref{assump:bounded_case} holds.
	For every $\tau_{\max} >0$,
	there exist constants
	$L^*, \varepsilon^* >0$ and $\tau_{\min}^* \in (0,\tau_{\max})$ such that 
	the system \eqref{eq:plant}
	with the event-triggering mechanism \eqref{eq:ETC_LB}
	is exponentially stable for 
	every $L \in [0,L^*]$, $\varepsilon \in [0,\varepsilon^*]$, and
	$\tau_{\min} \in (0,\tau_{\min}^*]$.
\end{theorem}

\begin{proof}
	In the proof, we first investigate
	the integral terms in the mild solution 
	\eqref{eq:x_TBF_rep} and
	obtain upper bounds of their norms $|\cdot|$ defined by \eqref{eq:new_norm_cont}. 
	Using these upper bounds, we next prove that 
	\begin{align}
	\label{eq:x_estimate_ETC}
	|x(t_k+\tau)| &\leq
	e^{-\gamma_0 \tau} |x(t_k)|\qquad \forall \tau \in [\tau_{\min},t_{k+1}-t_k],~\forall k\in \mathbb{N}_0
	\end{align}
	for some $\gamma_0 >0$.
	Finally, we show that 
	the event-triggered control system is exponentially stable, by using
	the above estimate \eqref{eq:x_estimate_ETC} and the properties \eqref{eq:new_norm_cont_prop} of the norm $|\cdot|$.
	
	1. 
	Under the event-triggering mechanism  \eqref{eq:ETC_LB}, the solution $x$
	of the integral equation~\eqref{eq:mild_solution} satisfies
	\[
	\|x(t_k) - x(t_k+s)\| \leq \varepsilon \|x(t_k)\|\qquad
	\forall s \in [\tau_{\min}, t_{k+1}-t_k),~\forall k \in \mathbb{N}_0.
	\]
	Note that the above inequality may  not hold for $s \in (0,\tau_{\min})$.
	Therefore, compared with the case of 
	the event-triggering mechanism \eqref{eq:ETM1_intro} 
	studied in the previous study \cite{Wakaiki2018_EVC},
	the careful estimate of the term in the mild solution \eqref{eq:x_TBF_rep}, 
	\[
	\int^\tau_{0} T_{BF}(\tau-s) BF 
	[
	x(t_k) - x(t_k+s)
	]ds ,
	\]
	is required.
	
	From the properties \eqref{eq:new_norm_cont_prop} of the norm $|\cdot|$,
	it follows that for every $\tau \in [\tau_{\min}, t_{k+1}-t_k]$,
	\begin{align}
	&\left|
	\int^\tau_{\tau_{\min}} T_{BF}(\tau-s) BF 
	[
	x(t_k) - x(t_k+s)
	]ds 
	\right| \notag \\
	& \qquad \qquad \leq 
	\Gamma \|BF\|_{\mathcal{B}(X)} \int^\tau_{\tau_{\min}} e^{-\gamma(\tau-s)} 
	\|
	x(t_k) - x(t_k+s)
	\| ds \notag \\
	&\qquad \qquad \leq
	b_{\rm e}\varepsilon \big(1-e^{-\gamma(\tau-\tau_{\min} )}\big) |x(t_k)|,
	\label{eq:tau_min_above}
	\end{align}
	where
	$
	b_{\rm e} :=  \Gamma \|BF\|_{\mathcal{B}(X)} /\gamma.
	$
	To estimate 
	\[
	\left|
	\int^{\tau_{\min}}_0 T_{BF}(\tau-s) BF 
	[
	x(t_k) - x(t_k+s)
	]ds 
	\right|,\qquad \tau_{\min} \leq  \tau \leq t_{k+1}-t_k,
	\]
	we define
	\begin{equation}
	\label{eq:c1_c2_def}
	c_1 := c_1(\tau_{\min}) = 
	\sup_{0\leq \tau \leq \tau_{\min}} \|\Delta_\tau\|,\quad 
	c_2 := c_2(\tau_{\min}) = \sup_{0\leq \tau \leq \tau_{\min}} \|T(\tau)\|.
	\end{equation}
	By \eqref{eq:x_Delta} and Gronwall's inequality, 
	\begin{align}
	\label{eq:x_bound_tau_small}
	\|x(t_k+\tau)\| \leq c_1 \|x(t_k)\| + c_2L \int^\tau_0 \|x(t_k+s)\| ds 
	\leq c_1 e^{c_2L\tau} \|x(t_k)\|
	\end{align}
	for every $\tau \in [0,\tau_{\min}]$ and every $k \in \mathbb{N}_0$.
	Therefore, using \eqref{eq:x_Delta} again, we obtain
	\begin{align*}
	\|BF[x(t_k) - x(t_k+\tau)]\| 
	&\leq 
	\left(
	\|BF(I-\Delta_\tau)\|_{\mathcal{B}(X)} + c_1\|BF\|_{\mathcal{B}(X)} 
	(e^{c_2L\tau} - 1)
	\right) \|x(t_k)\|.
	\end{align*}
	Define 
	\[
	g(L,\tau_{\min}) := \Gamma \sup_{0\leq \tau\leq \tau_{\min}}
	\left(\|BF(I-\Delta_\tau)\|_{\mathcal{B}(X)} 
	+ c_1\|BF\|_{\mathcal{B}(X)} 
	(e^{c_2L\tau} - 1) \right).
	\]
	The properties \eqref{eq:new_norm_cont_prop} of the norm $|\cdot|$ yield
	\begin{align}
	\label{eq:tau_min_below}
	\left| \int^{\tau_{\min}}_0
	T_{BF}(\tau-s) BF[ 
	x(t_k) - x(t_k+s)
	]ds
	\right| \leq \tau_{\min} g(L,\tau_{\min}) |x(t_k)|
	\end{align}
	for every $\tau \in [\tau_{\min}, t_{k+1}-t_k]$.
	
	To estimate the other integral term
	in the mild solution \eqref{eq:x_TBF_rep}, 
	\[
		\int^\tau_0 T_{BF}(\tau-s) \phi\big(
	x(t_k+s)
	\big)ds,
	\]
	first note that, 
	in the same way as in the proof of Lemma~\ref{lem:inter_sample_bounded_case},
	one can obtain 
	\[
	\|x(t_k+\tau)\| \leq 
	\Gamma e^{(\Gamma L - \gamma)(\tau- \tau_{\min}) } \|x(t_k+\tau_{\min})\|+ 
	\widetilde \eta_{L,\varepsilon} (\tau - \tau_{\min}) \|x(t_k)\| 
	\]
	for every $\tau \in [\tau_{\min}, t_{k+1}-t_k)$,
	where 
	\[
		\widetilde \eta_{L,\varepsilon} (\tau) := 
		\frac{
		\varepsilon \Gamma \|BF\|_{\mathcal{B}(X)}}{\gamma}
	\big(e^{\Gamma L \tau}- 
	e^{(\Gamma L - \gamma)\tau }\big).
	\]
	Combining this with \eqref{eq:x_bound_tau_small}, we obtain
	\begin{equation}
	\label{eq:x_bound_ETC}
	\|x(t_k+\tau)\| \leq \Upsilon_{L,\varepsilon,\tau_{\min}} (\tau) \|x(t_k)\|
	\qquad
	\forall \tau \in [0, t_{k+1}-t_k),
	\end{equation}
	where 
	\[
	\Upsilon_{L,\varepsilon,\tau_{\min}} (\tau) :=
	\begin{cases}
	c_1e^{c_2L\tau}, & 0\leq \tau \leq \tau_{\min}, \\
	c_1\Gamma e^{c_2L\tau_{\min}}
	e^{(\Gamma L - \gamma)(\tau- \tau_{\min}) }+
	\widetilde \eta_{L,\varepsilon}(\tau-\tau_{\min}), &  \tau \geq \tau_{\min}.
	\end{cases}
	\]
	Then 
	\begin{align}
	\label{eq:TBF_phi_ETM2}
	\left|
	\int^\tau_0 T_{BF}(\tau-s) \phi\big(
	x(t_k+s)
	\big)ds	
	\right| 
	&\leq \Gamma L \int^\tau_0 e^{-\gamma(\tau-s)} \Upsilon_{L,\varepsilon,\tau_{\min}}(s)ds |x(t_k)| \notag \\
	&\leq \beta_{\rm e}(L,\varepsilon,\tau_{\min},\tau_{\max}) (1-e^{-\gamma \tau}) |x(t_k)|,
	\end{align}
	for  every $\tau \in [\tau_{\min},t_{k+1}-t_k]$ and every $k \in \mathbb{N}_0$,
	where 
	\begin{equation}
	\label{eq:alpha_def}
	\beta_{\rm e}(L,\varepsilon,\tau_{\min},\tau_{\max}) := \sup_{\tau_{\min} \leq \tau \leq \tau_{\max}}
	\frac{\Gamma L }{1-e^{-\gamma \tau}}\int^\tau_0 e^{-\gamma(\tau-s)} \Upsilon_{L,\varepsilon,\tau_{\min}}(s)ds.
	\end{equation}
	
	2. Combining \eqref{eq:x_TBF_rep} with
	the estimates \eqref{eq:tau_min_above}, \eqref{eq:tau_min_below}, and 
	\eqref{eq:TBF_phi_ETM2}, we obtain
	\begin{align}
	\label{eq:x_bound_ETM2}
	|x(t_k+\tau)| &\leq
	\nu(\tau) |x(t_k)|\qquad \forall \tau \in [\tau_{\min},t_{k+1}-t_k],~\forall k\in \mathbb{N}_0,
	\end{align}
	where
	\begin{equation}
	\label{eq:nu_def}
	\nu(\tau):=
	e^{-\gamma \tau} +\tau_{\min} g(L,\tau_{\min}) + 
	\beta_{\rm e}(L,\varepsilon,\tau_{\min},\tau_{\max})
	(
	1-e^{-\gamma \tau}
	) + 
	b_{\rm e}\varepsilon 
	\big( 1 - e^{-\gamma (\tau-\tau_{\min})}
	\big).
	\end{equation}
	We will prove $\sup_{\tau \geq \tau_{\min}}\nu(\tau) < 1$.
	To this end, define
	\begin{align*}
	\kappa_1(\tau) &:= e^{-\gamma \tau} + \tau g(L,\tau) \\
	\kappa_2(\tau) &:= e^{-\gamma \tau} - e^{-\gamma \tau_{\min}}+ 
	\beta_{\rm e}(L,\varepsilon,\tau_{\min},\tau_{\max})
	(
	1-e^{-\gamma \tau}
	) +
	b_{\rm e}\varepsilon 
	\big( 1 - e^{-\gamma (\tau-\tau_{\min})}
	\big).
	\end{align*}
	Then $\nu(\tau) = \kappa_1(\tau_{\min}) + \kappa_2(\tau)$.
	
	First we investigate $\kappa_1(\tau_{\min})$.
	Since $e^{-\gamma \tau} < 1 - \gamma \tau e^{-\gamma \tau}$
	for every $\tau >0$,
	it follows that 
	\begin{align*}
	\kappa_1(\tau)
	<
	1-\gamma \tau + 
	[
	\gamma (1-e^{-\gamma \tau}) + g(L,\tau)
	]\tau\qquad \forall \tau>0.
	\end{align*}
	Lemma~\ref{lem:FD_conv} shows that  for every $L>0$,
	\begin{equation}
	\label{eq:g_0_conv}
	\lim_{\tau_{\min} \downarrow 0}g(L,\tau_{\min}) = 0.
	\end{equation}
		Choose $\varsigma \in (0,\gamma)$ arbitrarily, and let $L^*_0 >0$. There exists
	$\tau_{\min}^* \in (0,\tau_{\max})$ such that 
	\begin{equation}
	\label{eq:ETC_cond1}
	\gamma (1-e^{-\gamma \tau_{\min}^*}) + g(L^*_0,\tau_{\min}^*) \leq \varsigma.
	\end{equation}
	Let $\tau_{\min} \in (0,\tau_{\min}^*]$ be given. 
	For every $L \in [0,L^*_0]$,
	\begin{equation}
	\label{eq:bound1_ETM2}
	\kappa_1(\tau_{\min}) < 1 - (\gamma -\varsigma ) \tau_{\min}.
	\end{equation}
	
	To estimate $\kappa_2(\tau)$,
	we first obtain
	\[
	\kappa_2'(\tau) = \gamma \left(
	\beta_{\rm e}(L,\varepsilon,\tau_{\min},\tau_{\max}) 
	+ b_{\rm e} \varepsilon e^{\gamma \tau_{\min}} - 1
	\right)e^{-\gamma \tau} .
	\]
	Since
	\[
	\lim_{\tau \downarrow 0}
	\frac{1}{1-e^{-\gamma \tau}}\int^\tau_0 e^{-\gamma(\tau-s)} \Upsilon_{L,\varepsilon,\tau_{\min}}(s)ds = \frac{c_1}{\gamma},
	\]
	it follows from
	the definition \eqref{eq:alpha_def} of $\beta_{\rm e}$ that 
	there exist $L^* \in (0,L^*_0]$ and $\varepsilon^*>0$, 
	independent of $\tau_{\min}$, such that 
	\begin{align}
	\label{eq:ETC_cond2}
	\sup_{0 < \tau \leq \tau_{\min}^*}
	\left(
	\beta_{\rm e}(L^*,\varepsilon^*,\tau,\tau_{\max})  
	+ b_{\rm e} \varepsilon^* e^{\gamma \tau} 
	\right)
	< \frac{\gamma - \varsigma}{\gamma} 
	~(< 1).
	\end{align}
	Choose $L \in [0,L^*]$ and $\varepsilon \in [0,\varepsilon^*]$.
	Then $\kappa_2'(\tau) <0$ for every $\tau >0$, and hence
	\begin{equation}
	\label{eq:bound2_ETM2}
	\kappa_2(\tau) \leq \kappa_2(\tau_{\min}) = 
	\beta_{\rm e}(L,\varepsilon,\tau_{\min},\tau_{\max})
	(
	1-e^{-\gamma \tau_{\min}}
	)\qquad \forall \tau \geq \tau_{\min}.
	\end{equation}
	By \eqref{eq:bound1_ETM2} and \eqref{eq:bound2_ETM2},
	\[
	\nu(\tau) 
	\leq \kappa_1(\tau_{\min}) + \kappa_2(\tau_{\min})
	< 1 - (\gamma - \varsigma)\tau_{\min} + 
	\beta_{\rm e}(L,\varepsilon,\tau_{\min},\tau_{\max})
	(
	1-e^{-\gamma \tau_{\min}}
	)
	\]
	for every $\tau \geq \tau_{\min}$.
	If we 
	define a function $w$ on $(0,\infty)$ by
	\[
	w(\tau) := \frac{(\gamma - \varsigma)\tau}{1- e^{-\gamma \tau}},
	\]
	then $w$ is increasing on $(0,\infty)$ and $\lim_{\tau \downarrow 0}w(\tau) 
	= (\gamma - \varsigma)/\gamma$.
	Since \eqref{eq:ETC_cond2} yields
	\[
	\beta_{\rm e}(L,\varepsilon,\tau_{\min},\tau_{\max}) < 
		\frac{\gamma - \varsigma}{\gamma}  <
	\frac{ (\gamma - \varsigma)\tau_{\min} }{1-e^{-\gamma \tau_{\min}}},
	\]
	it follows that 
	$\sup_{\tau \geq \tau_{\min}}\nu(\tau) < 1$.
	Hence
	\[
	\gamma_0 := \inf_{\tau_{\min} \leq \tau \leq \tau_{\max}} 
	\frac{-\log \nu(\tau)}{\tau} >0,
	\] 
	and \eqref{eq:x_bound_ETM2} yields 
	\begin{equation}
	\label{eq:x_bound_tau_large}
	|x(t_k+\tau)| \leq e^{-\gamma_0\tau} |x(t_k)|\qquad \forall \tau \in [\tau_{\min},t_{k+1}-t_k],~\forall k\in \mathbb{N}_0.
	\end{equation}
	
	3.
	Substituting $\tau = t_{k+1}- t_k$ into \eqref{eq:x_bound_tau_large} gives
	\[
	|x(t_{k+1})| \leq e^{-\gamma_0 (t_{k+1} - t_k)} |x(t_k)|\qquad \forall k\in \mathbb{N}_0.
	\]
	Applying induction, we obtain
	\begin{align*}
	|x(t_k)| &\leq e^{-\gamma_0t_k} |x^0| \qquad 
	\forall x^0 \in X,~\forall k\in \mathbb{N}_0.
	\end{align*}
	By \eqref{eq:x_bound_tau_small} and \eqref{eq:x_bound_tau_large},
	there exists $M \geq 1$ such that
	\[
	\|x(t_k+\tau)\| \leq M\|x(t_k) \| \qquad \forall \tau \in [0,t_{k+1}-t_k],~\forall k\in \mathbb{N}_0.
	\]
	Therefore,
	\[
	\|x(t_k+\tau)\| \leq 
	M \|x(t_k)\| \leq M|x(t_k)| \leq 
	M e^{-\gamma_0 t_k} |x^0| \leq 
	\big(M\Gamma e^{\gamma_0 \tau_{\max}}  \big)  e^{-\gamma_0 (t_k+\tau)} \|x^0\| 
	\]
	for all $x^0 \in X$, $\tau \in [0,t_{k+1}-t_k]$,~and $k\in \mathbb{N}_0$.
	Thus, the event-triggered control system is exponentially stable.
\end{proof}

\begin{remark}[Conditions on $L^*$, $\varepsilon^*$, and $\tau_{\min}^*$]
	\label{rem:ETC_cond_summary}
	We see from \eqref{eq:ETC_cond1} and \eqref{eq:ETC_cond2} that,
	for a given $\tau_{\max} >0$,
	the bounds 
	$L^*, \varepsilon^* >0$ and $\tau_{\min}^* \in (0,\tau_{\max})$
	in Theorem~\ref{thm:ETC_LB_stability}
	have to satisfy the following two inequalities:
		\begin{align*}
		&\gamma 
		e^{-\gamma \tau_{\min}^*} - \Gamma 
		\sup_{0\leq \tau \leq \tau^*_{\min}} \Big(
		\|BF(I-\Delta_\tau)\|_{\mathcal{B}(X)} + c_1(\tau_{\min}^*)\|BF\|_{\mathcal{B}(X)}  (e^{c_2(\tau_{\min}^*)L^*\tau} - 1)\Big) \\
		&=: \varsigma_1 >0    \\
		&
		\sup_{0< \tau \leq \tau_{\min}^*}
		\left(
		\beta_{\rm e}(L^*,\varepsilon^*,\tau,\tau_{\max})  
		+
		\frac{\varepsilon^* e^{\gamma \tau} 
			\Gamma \|BF\|_{\mathcal{B}(X)}}{\gamma} 
		\right)
		< \frac{\varsigma_1}{\gamma},
		\end{align*}
		where we define
		$c_1(\tau_{\min})$ and  $c_2(\tau_{\min})$ by \eqref{eq:c1_c2_def}
		 and 
		$\beta_{\rm e}(L,\varepsilon,\tau_{\min},\tau_{\max})$ by
		\eqref{eq:alpha_def}.
\end{remark}

The conditions given in Remark~\ref{rem:ETC_cond_summary} 
look complicated.
However, in the unperturbed case  $\phi \equiv 0$, 
we obtain a simple sufficient condition for exponential stability, which
can be used for the design of the event-triggering mechanism \eqref{eq:ETC_LB}.
Here we assume that $BF \not=0$; otherwise
$T(t)$ is exponentially stable under Assumption~\ref{assump:bounded_case},
and hence the stabilization problem we consider would be trivial.
\begin{corollary}
	\label{coro:ETC_LB_stability}
	Let Assumption~\ref{assump:bounded_case}, $BF \not=0$, and
	$\phi \equiv 0$ be satisfied.
	If $\varepsilon\geq0$ and $\tau_{\min} \in (0,\tau_{\max})$ 
	satisfy
	\begin{equation}
	\label{eq:linear_case_ETC}
	\varepsilon < \frac{
	\gamma e^{-\gamma \tau_{\min}} 
		-
		\Gamma\sup_{0\leq \tau \leq \tau_{\min}} 
		\|BF(I-\Delta_\tau)\|_{\mathcal{B}(X)}}{e^{\gamma \tau_{\min}}\Gamma \|BF\|_{\mathcal{B}(X)} } ,
	\end{equation}
	the system \eqref{eq:plant}
	with the event-triggering mechanism \eqref{eq:ETC_LB}
	is exponentially stable for 
	every $\tau_{\max} >0$.
\end{corollary}
\begin{proof}
	Substituting $L^*=0$ into the conditions 
	in Remark~\ref{rem:ETC_cond_summary},
	we obtain the condition \eqref{eq:linear_case_ETC}.
\end{proof}

\begin{remark}[Dependence on $\tau_{\max}$]
In \eqref{eq:linear_case_ETC}, $\tau_{\max}$ does not appear. However,
the decay rate  of the closed-loop system may become small
as $\tau_{\max}$ increases, as in the self-triggered case.
\end{remark}
\begin{remark}[Robustness to linear perturbations]
	\label{rem:robustness_of_ETC}
	Note that the event-triggering mechanism \eqref{eq:ETC_LB} may allow
	larger linear perturbations than the self-triggered mechanism
	\eqref{eq:STC_Feedback}.
	The reason is that 
	the event-triggering mechanism \eqref{eq:ETC_LB} 
	does not use the model of the plant. To see this,
	suppose that the plant $(A,B)$ is changed to 
	$(\widetilde A, \widetilde B)$, where 
	$\widetilde A$ is the generator of  a strongly continuous semigroup
	on $X$ and $\widetilde B \in \mathcal{B}(U,X)$.
	The perturbed event-triggering control system 
	is exponentially stable as long as 
	$\varepsilon>0$ and $\tau_{\min} \in (0,\tau_{\max})$
	satisfies
	the counterpart of \eqref{eq:linear_case_ETC} in the
	perturbed case $(\widetilde A, \widetilde B)$.
	We observe
	this robustness of 
	the event-triggering mechanism \eqref{eq:ETC_LB} 
	against linear perturbations
	from numerical simulations in the next section
\end{remark}

\begin{remark}[Periodic case]
	\label{rem:periodic_case}
	Consider the unperturbed periodic sampled-data system, that is,
	the case $\phi \equiv 0$ and $t_{k+1}-t_k \equiv h$.
	In the proof of Theorem~3.1 of \cite{Logemann2003},
	the following 
	sufficient condition for the periodic sample-data system to
	be exponentially stable  is provided under 
	Assumption~\ref{assump:bounded_case}:
	\begin{equation}
	\label{eq:suff_cond1_PSD}
	\Gamma e^{\gamma h}\sup_{0 \leq t\leq h} \big\|
	T(h-t) BF [I-T_{BF}(t)]\big\|_{\mathcal{B}(X)} < \gamma .
	\end{equation}
	From \eqref{eq:linear_case_ETC} with $\varepsilon = 0$,
	we also obtain a sufficient condition 
	\begin{equation}
	\label{eq:suff_cond2_PSD}
	\Gamma e^{\gamma h}\sup_{0 \leq t\leq h} \|
	BF (I-\Delta_t)\|_{\mathcal{B}(X)} < \gamma
	\end{equation}
	for the periodic sample-data system with $t_{k+1}-t_k \equiv h$ 
	to
	be exponentially stable. 
	These sufficient conditions \eqref{eq:suff_cond1_PSD}
	and \eqref{eq:suff_cond2_PSD}
	are essentially same, because
	the technique used to prove
	Theorem~3.1 of \cite{Logemann2003}
	is applied 
	for the inequality \eqref{eq:bound1_ETM2} in the proof of
	Theorem~\ref{thm:ETC_LB_stability}.
\end{remark}
\section{Numerical example in bounded control case}
\label{sec:numerical_example}
In this section, we provide numerical simulations of 
the event/self-triggering mechanisms studied in Section~\ref{sec:bounded_control}.
Before presenting with simulation results,
we explain the applicability of the proposed methods.
Recall 
that the semigroup $T_{BF}(t)$ generated by $A+BF$ is exponentially stable
under Assumption~2.4.
Exponential stabilization by a compact feedback operator $F$
is achieved in the bounded control case only if
$A$ has only finitely many unstable eigenvalues; see
Theorem~IV.8.24 on p.~469 of \cite{Engel2000}.
Therefore, the proposed methods can be applied to
systems with finitely many unstable poles such as
heat equations and retarded delay differential equations, but
not to 
systems with infinitely many unstable poles such as undamped wave equations.

\subsection{Heat equation  in cascaded with ODE}
We consider a heat equation in cascade with an ODE:
\begin{subequations}
	\label{eq:ODE_PDE}
	\noeqref{eq:ODE_PDE1,eq:ODE_PDE2,eq:ODE_PDE3}
	\begin{align}
	&\frac{\partial z_1}{\partial t}(\xi,t) = 
	\frac{\partial^2 z_1}{\partial \xi^2}(\xi,t) + b(\xi)
	\psi\big(z_2(t)\big),\quad \xi \in [0,1],~t \geq 0 \label{eq:ODE_PDE1}\\
	& \frac{\partial z_1}{\partial \xi}(0,t) = 0,\quad \frac{\partial z_1}{\partial \xi}(1,t) = 0,\quad t\geq 0;\quad
	z_1(\xi,0) = z_1^0(\xi),\quad \xi \in [0,1] \label{eq:ODE_PDE2}\\
	& \dot z_2(t) = Gz_2(t) + Hu(t),\quad t\geq 0;\qquad
	z_2(0) = z_2^0,\label{eq:ODE_PDE3}
	\end{align}
\end{subequations}
where $b = 
\begin{bmatrix}
b_1 & \cdots & b_p
\end{bmatrix}\in L^2([0,1],\mathbb{R}^{1 \times p})$,
$G \in \mathbb{R}^{p \times p}$, and $H \in \mathbb{R}^{p \times m}$.
In \eqref{eq:ODE_PDE},
$z_1(\xi,t) \in \mathbb{R}$ is the temperature at position $\xi \in [0,1]$ and time $t\geq0$,
$z_2(t)\in \mathbb{R}^p$ is the state of the ODE, 
$u(t)\in \mathbb{R}^m$ is the input, and $\psi: \mathbb{R}^p \to \mathbb{R}^p$
represents the actuator nonlinearity of the heat equation.

First
we reformulate the cascaded system \eqref{eq:ODE_PDE} as 
an abstract evolution equation in the form of \eqref{eq:state_equation}.
We write $L^2(0,1)$ in place of $L^2([0,1],\mathbb{C})$.
The state space $X$ and the input space $U$ are defined by
$X :=   L^2(0,1) \times \mathbb{C}^p$ and $U := \mathbb{C}^m$,
respectively.
The state space $X$ is a Hilbert space endowed with the inner product
\[
\left\langle
\begin{bmatrix}
x_1 \\ x_2
\end{bmatrix},~
\begin{bmatrix}
y_1 \\ y_2
\end{bmatrix}
\right\rangle :=
\langle x_1,
y_1
\rangle_{L^2} + 
\left\langle
x_2,
y_2
\right\rangle_{\mathbb{C}^p}.
\]
Set
\[
x(t) :=  
\begin{bmatrix}
x_1(t) \\ x_2(t)
\end{bmatrix}
\text{~~with~~$x_1(t) := z_1(\cdot, t)$ and $x_2(t) := z_2(t)$};
\qquad x^0 :=
\begin{bmatrix}
z_1^0 \\ z_2^0
\end{bmatrix} \in X.
\]
Let $f_0(\xi) := 1$ and $f_n (\xi) := \sqrt{2} \cos(n \pi\xi )$ for 
$n \in \mathbb{N}$. Then $\{f_n\}_{n \in \mathbb{N}_0}$
forms an orthonormal basis for $L^2(0,1)$.
Define 
$A_1 : \dom(A_1) \subset L^2(0,1) \to L^2(0,1)$ by
\begin{equation}
\label{eq:A1_expression}
A_1x_1 :=-\sum_{n = 0}^\infty 
n^2\pi^2
\langle
x_1,f_n
\rangle_{L^2} f_{n} 
\end{equation}
with domain
\[
\dom(A_1) := \left\{
x_1 \in L^2(0,1):
\sum_{n = 0}^\infty  n^4\pi^4 
|\langle
x_1,f_n
\rangle_{L^2}|^2 < \infty
\right\}
\]
and $B_1 : \mathbb{C}^p \to L^2(0,1)$  by
\[
B_1x_2 :=  bx_2,\quad x_2 \in \mathbb{C}^p.
\]
As shown  in Example~2.3.7 on p.~45 in \cite{Curtain1995},
$A_1$ given in \eqref{eq:A1_expression}  is the operator
that governs the state evolution of the uncontrolled heat equation
with Neumann boundary conditions.
If we define the operators $A:\dom(A) \subset X \to X$, 
$B:U \to X$, and $\phi:X \to X$ by
\begin{align*}
A &:= 
\begin{bmatrix}
A_1 & B_1 \\
0 & G
\end{bmatrix}\quad \text{with~~$\dom(A) := \dom(A_1) \times \mathbb{C}^p$}\\
B &:= 
\begin{bmatrix}
0 \\
H
\end{bmatrix},\quad
\phi
\left(
\begin{bmatrix}
x_1 \\ x_2
\end{bmatrix}
\right) := 
\begin{bmatrix}
B_1\psi(x_2) - B_1x_2 \\ 0
\end{bmatrix},
\end{align*}
then the cascaded system \eqref{eq:ODE_PDE} is reformulated as
an abstract evolution equation \eqref{eq:state_equation}.

Let the feedback operator $F: X \to U$ be in the form of
\begin{equation}
\label{eq:feedback_operator_example}
F \begin{bmatrix}
x_1 \\ x_2
\end{bmatrix} := F_1 \langle x_1, f_0 \rangle + F_2 x_2,
\end{equation}
where $F_1 \in \mathbb{C}^m$ and $F_2 \in \mathbb{C}^{m \times p}$.
By construction,
the controller uses 
the average temperature $\langle x_1(t), f_0\rangle$
for the computation of the control input $u(t)$.
Similarly, the 
self-triggering mechanism \eqref{eq:STC_Feedback}
computes
the transmission time $t_{k+1}$
from  the average temperature $\langle x_1(t_k), f_0\rangle$
and the $L^2$ norm $\|x_1(t_k)\|_{L^2}$.
In fact,
a simple calculation shows that
\begin{subequations}
	\label{eq:FT_FST_matrix_rep}
	\noeqref{eq:FT_FST_matrix_rep1, eq:FT_FST_matrix_rep2}
	\begin{align}
	FT(t) 
	\begin{bmatrix}
	x_1 \\ x_2
	\end{bmatrix}
	&=
	\begin{bmatrix}
	F_1 & Q(t)
	\end{bmatrix}
	\begin{bmatrix}
	\langle x_1, f_0 \rangle \\ x_2
	\end{bmatrix} \label{eq:FT_FST_matrix_rep1}\\
	FS_tF
	\begin{bmatrix}
	x_1 \\ x_2
	\end{bmatrix}
	&=
	\begin{bmatrix}
	\int^t_0 Q(s)H F_1 ds & \int^t_0 Q(s)H F_2 ds
	\end{bmatrix}
	\begin{bmatrix}
	\langle x_1, f_0 \rangle \\ x_2
	\end{bmatrix}\label{eq:FT_FST_matrix_rep2}
	\end{align}
\end{subequations}
for every $t \geq 0$,
where $Q(t) \in \mathbb{C}^{m \times p}$ is defined by
\[
Q(t) :=
F_1 
\begin{bmatrix}
\langle b_1, f_0 \rangle_{L^2} & \cdots &  \langle b_p, f_0 \rangle_{L^2} 
\end{bmatrix}
 \int^t_0 e^{G s}  ds + F_2 e^{G t},\quad  t \geq 0.
\]
Moreover, this implies that the time sequence 
$\{t_k\}_{k \in \mathbb{N}_0}$ 
of the self-triggering mechanism \eqref{eq:STC_Feedback}
can be calculated by matrix operations.

When the prespecified lower bound $\tau_{\min}$ of inter-event times is 
$0$ in
the event-triggering mechanism \eqref{eq:ETC_LB},
an inter-event time can be made arbitrarily close to $0$ in this example.
To see this, let $\varepsilon >0$ and define
\[
t_1(x^0) :=  \inf \big\{ t> 0:
\|x^0 - x(t)\| > \varepsilon \|x^0\|\},\quad x^0 \in X.
\]
Then 
\[
\lim_{n \to \infty} t_1 \left( 
\begin{bmatrix}
f_n \\ 0
\end{bmatrix}
\right) = 0.
\]

\subsection{Numerical simulation: Self-triggered control}
\label{sec:simualtion_STC}
Let $p=m=1$ and
\begin{equation}
	\label{eq:constant_parameter}
b = b_1 = 5 \mathds{1}_{[0.4,0.6]},~~G = 1,~~H=1,~~F_1 = -4,~~F_2 = -5,
\end{equation}
where $\mathds{1}_{[0.4,0.6]}$ is 
the indicator function of the interval $[0.4,0.6]$.
By Proposition~6.1 and its proof of \cite{Wakaiki2018_EVC},
for every $\gamma \in (0,2)$,
there exists $\Gamma \geq 1$ such that 
\begin{equation}
\label{eq:M_bound}
\|T_{BF}(t)\|_{\mathcal{B}(X)}  \leq \Gamma e^{-\gamma t} \qquad \forall t\geq 0.
\end{equation}
We see that $\Gamma = 1.92$ and $\gamma = 1$
satisfies \eqref{eq:M_bound} from
numerical computation based on
the eigenfunction decomposition by $\{f_n\}_{n \in \mathbb{N}_0}$
as in Section~6 of \cite{Wakaiki2018_EVC}.
The initial states $z_1(\xi,0)$ and $z_2(0)$ are  given by 
$z_1(\xi,0) \equiv 2$ and $z_2(0) = -2$.

In the simulation,
the actuator nonlinearity $\psi:\mathbb{R} \to \mathbb{R}$ 
of the heat equation 
is given by
\[
\psi(x) := 
\begin{cases}
(1+r_1) x, & \text{$0 \leq x < \vartheta$}, \\
(1-r_2) x + (r_1+r_2) \vartheta, & \text{$x \geq \vartheta$},\\
-\psi(-x), & \text{$x < 0$}
\end{cases}
\]
for $\vartheta >0$ and $0 < r_1,r_2 < 1$.
This nonlinearity $\psi$
represents
the energy efficiency of the actuator of the heat equation.
The actuator efficiency takes the higher value $1+r_1$ 
than the nominal value $1$ for
small inputs but the lower value $1-r_2$ for large inputs.
The constant $\vartheta$ is the threshold of 
the actuator efficiency.

Define $\psi_0(x) := \psi(x) - x$ for $x \in \mathbb{R}$ and
$r := \max\{r_1,r_2 \} $. Then
$|\psi_0(x) - \psi_0(y)| \leq r
|x - y|$ for every $x,y
\in \mathbb{R}$.
Hence the Lipschitz constant $L$ 
of $\phi$  is given by
$L = r\|B_1\|_{\mathcal{B}(\mathbb{C},L^2(0,1))} = \sqrt{5}r$, which 
does not depend on the threshold  $\vartheta$ because we consider a global 
Lipschitz condition. Therefore, 
we do not need to know
the exact value of the threshold $\vartheta $ for the design of the event/self-triggering mechanisms.
We set $\vartheta = 0.5$  in the simulations below.

Figures~\ref{fig:state_norm} and \ref{fig:input} show
the state norm $\|x(t)\|$ and 
the input $u(t)$ of the self-triggered control system with the 
mechanism \eqref{eq:STC_Feedback}, respectively.
We set $\tau_{\max} = 0.5$ and
consider the large perturbation case $(r_1,r_2,\varepsilon) = 
(0.1,0.1,0.29)$ and the small perturbation case 
$(r_1,r_2,\varepsilon) = 
(0.05,0.05,0.40)$, both of which satisfy the sufficient condition
\eqref{eq:L_thr_cond_STC} for exponential stability.
The blue and red lines in Figures~\ref{fig:state_norm} and \ref{fig:input} depict
the time responses in the large perturbation case and 
the small perturbation case, respectively.
We see from Figure~\ref{fig:state_norm} that 
the convergence speed of the state norm in
 the large perturbation case is slower, which
 is mainly due to the low efficiency of
the actuator of the heat equation.
Moreover, by looking at the update behavior on $[0.3,0.5]$ in 
Figure~\ref{fig:input}, we find that
the self-triggering mechanism in the large perturbation case 
is 
conservative, i.e., the control input is
updated even when the difference $Fx(t_{k+1}) - Fx(t_{k})$
is small. It is worthwhile to mention that if $\tau_{\max} > 0.91$, then 
$(r_1,r_2,\varepsilon) = 
(0.1,0.1,0.29)$ does not satisfy the sufficient condition
\eqref{eq:L_thr_cond_STC}.

\begin{figure}[bt]
	\centering
	\includegraphics[width = 8.8cm]{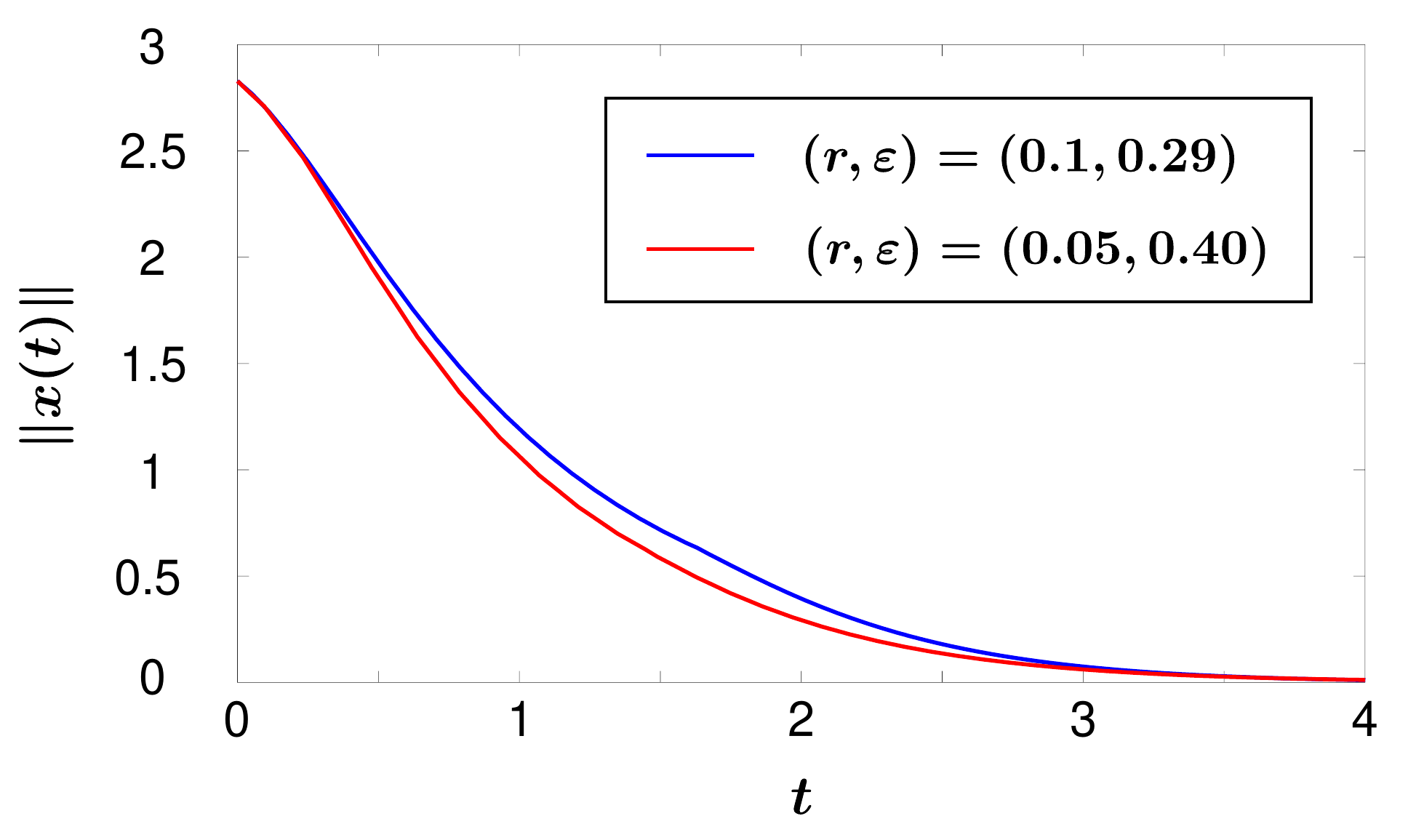}
	\caption{State norm $\|x(t)\|$ of self-triggered control system.}
	\label{fig:state_norm}
\end{figure}
\begin{figure}[bt]
	\centering
	\includegraphics[width = 8.8cm]{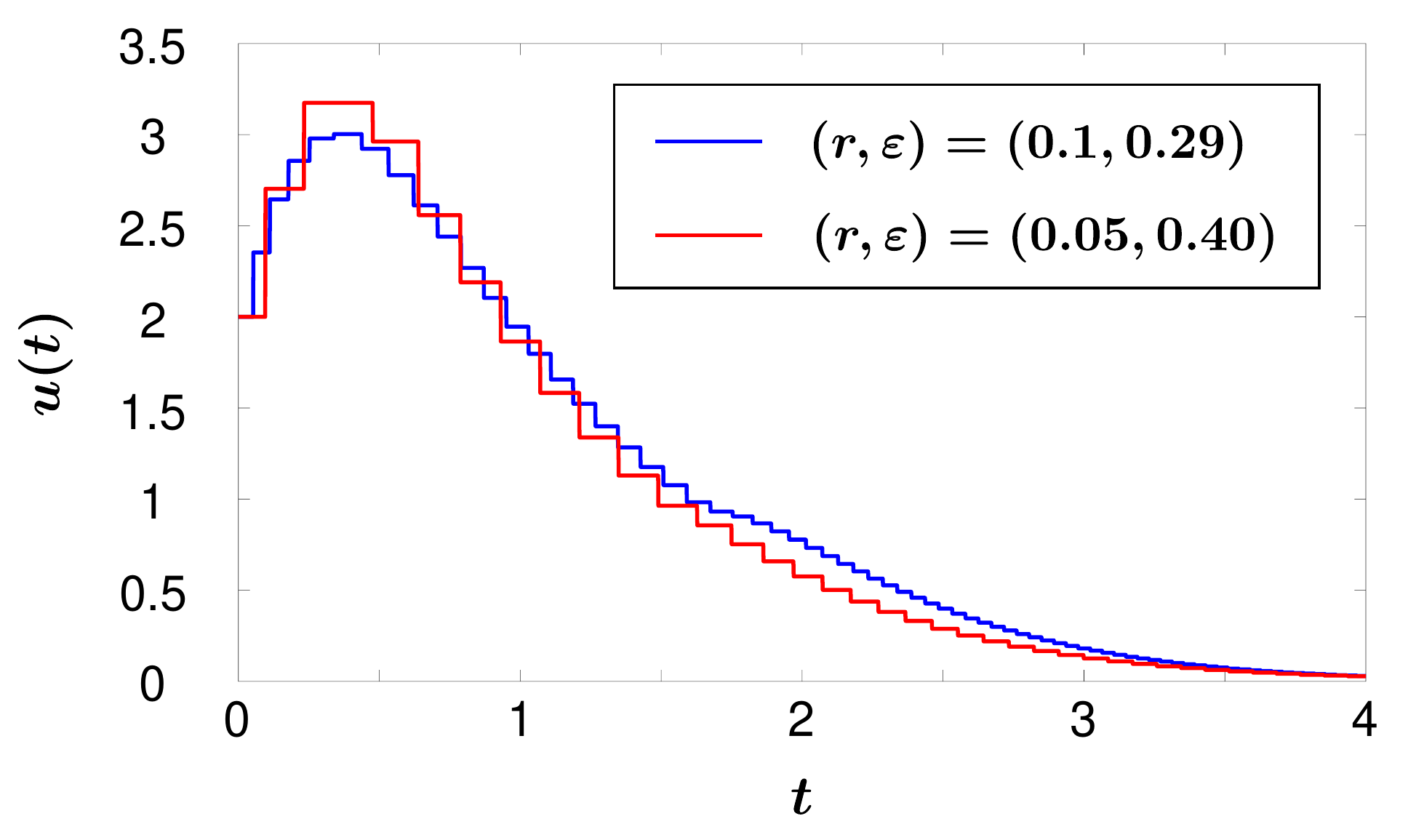}
	\caption{Input $u(t)$ of self-triggered control system.}
	\label{fig:input}
\end{figure}

Figure~\ref{fig:updating_interval} illustrates
inter-event times $t_{k+1} - t_k$ in the large perturbation case
 (blue circle) and  the small perturbation case 
(red squire).
The following lower bounds $\theta$ of the minimum inter-event time
$\inf_{k \in \mathbb{N}_0} (t_{k+1} - t_k)$ can be computed:
\[
\inf_{k \in \mathbb{N}_0} (t_{k+1} - t_k) \geq 
\theta:=
\inf_{\tau \geq 0}
\left\{
\|F(I-\Delta_{\tau})\| + 
L
\int_0^\tau \|FT(\tau - s)\| \eta_{L,\varepsilon}(s) ds \geq 
\varepsilon
\right\}.
\]
Indeed,
using  
the matrix representations
\eqref{eq:FT_FST_matrix_rep},
we obtain
$\theta = 0.0102$
in the large perturbation case  and
$\theta = 0.0147$ in 
the small perturbation case.
\begin{figure}[tb]
	\centering
	\includegraphics[width = 8.8cm]{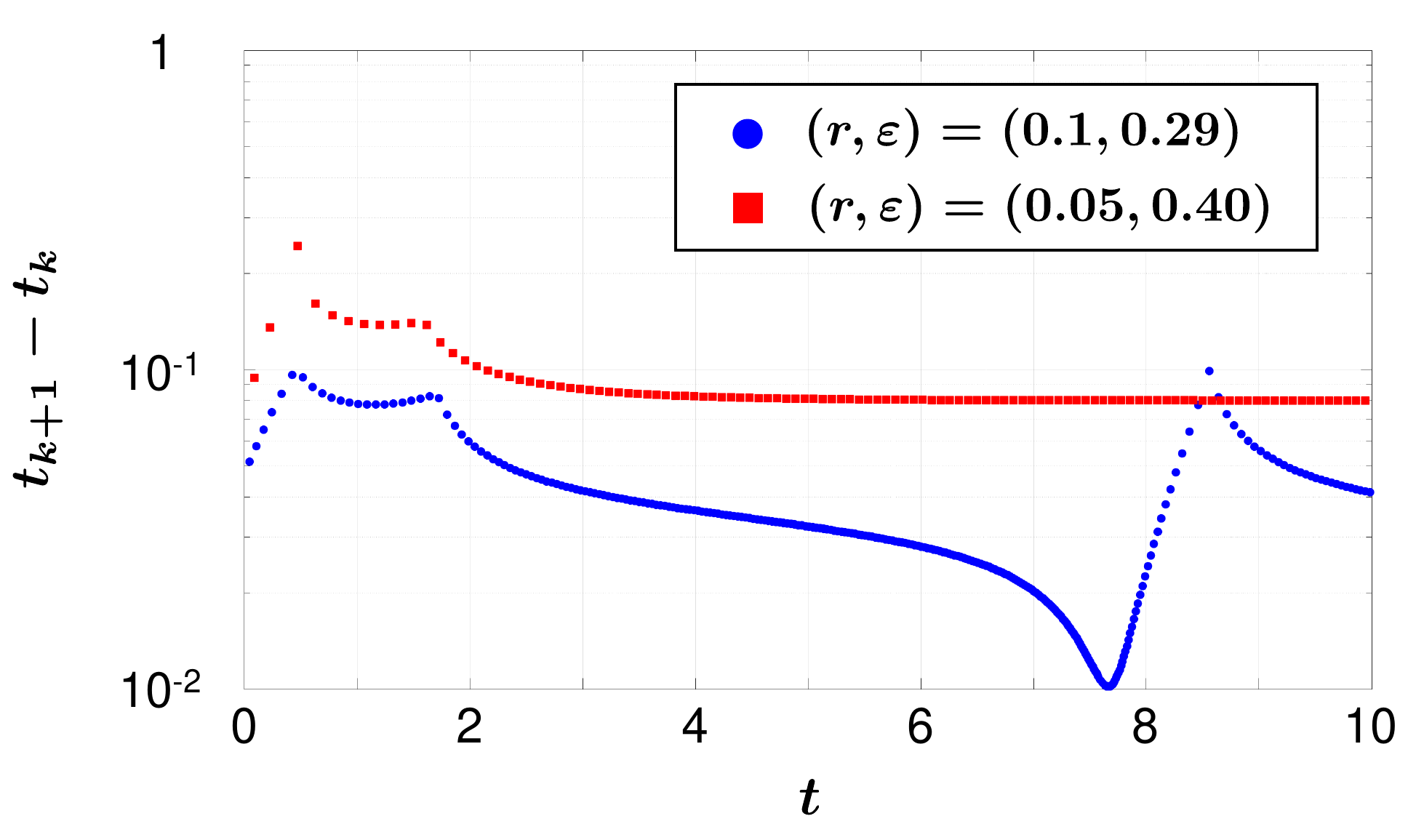}
	\caption{Inter-event times $t_{k+1} - t_k$ of self-triggered control system.}
	\label{fig:updating_interval}
\end{figure}
As expected from Figure~\ref{fig:input},
the state is frequently transmitted in the large perturbation case.  Moreover,
the inter-event times
in the large perturbation case
take values close to the lower bound $\theta = 0.0102$ for $7.6 \leq t \leq 7.8$.
We see from Figure~\ref{fig:updating_interval} that the behaviors of inter-event times
in the two cases are different.
More specifically, 
inter-event times in the large perturbation case
have a  periodic nature, and those in the small perturbation case
converge.
The understanding of the behaviors of inter-event times
remains limited even for finite-dimensional linear systems;
see, e.g., \cite{Postoyan2019}.
Further investigation would be needed to establish the analysis 
of inter-event times.

\subsection{Numerical simulation: Event-triggered control}
\label{sec:Num_Sim_ETC}
\begin{figure}[tb]
	\centering
	\includegraphics[width = 8.8cm]{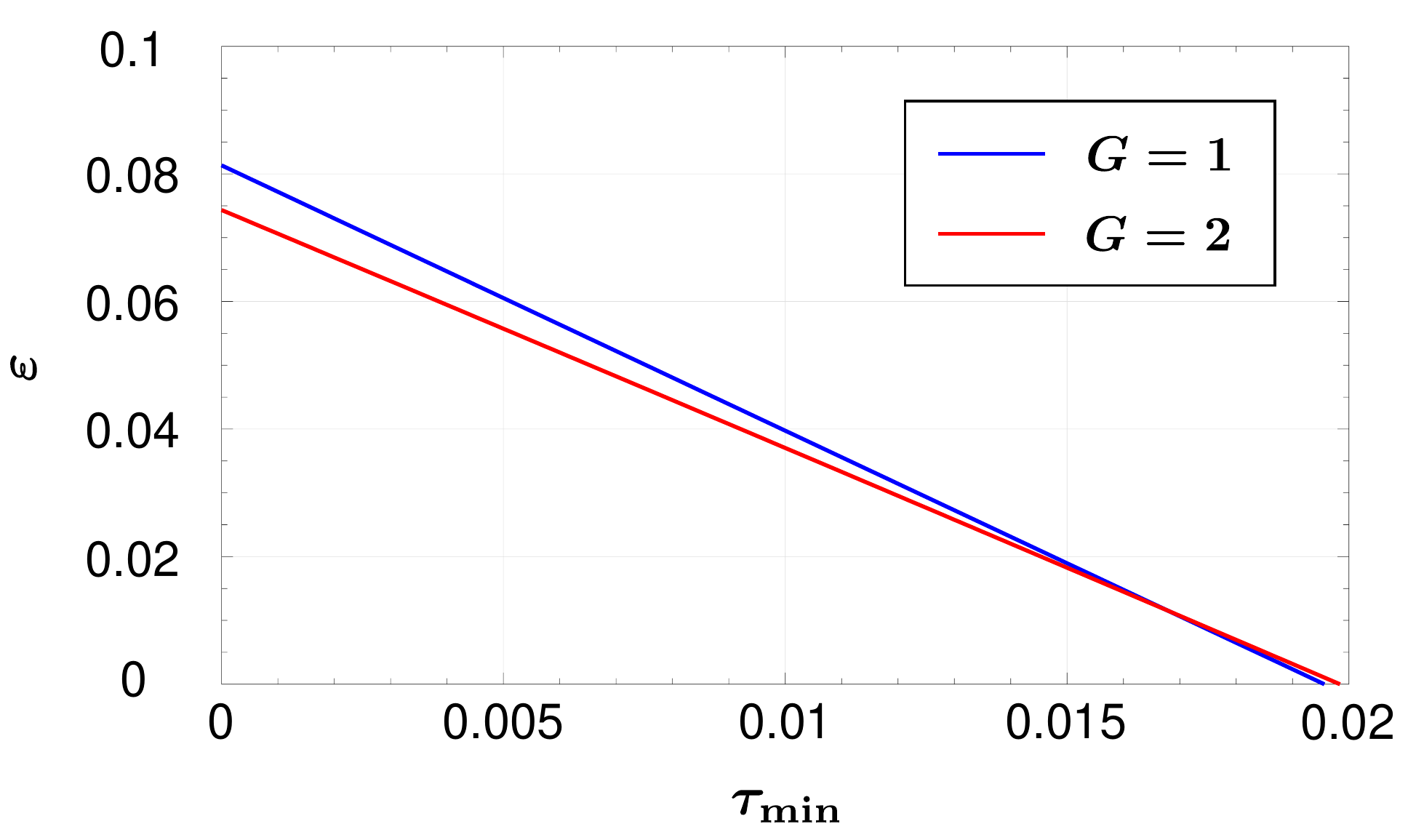}
	\caption{Bound on threshold $\varepsilon$ 
		and lower bound $\tau_{\min}$ of inter-event times of 
		event-triggering mechanism~\eqref{eq:ETC_LB}.}
	\label{fig:ETC_cond}
\end{figure}
Next, we provide numerical simulations of the event-triggering mechanism \eqref{eq:ETC_LB}.
To see the robustness against linear perturbations
explained in Remark~\ref{rem:robustness_of_ETC},
we here do not consider the nonlinear perturbation, that is,
we assume that $\psi(x_2) = x_2$ for every $x_2 \in \mathbb{R}$.
Figure~\ref{fig:ETC_cond} shows 
bounds on the threshold $\varepsilon$ and
the lower bound $\tau_{\min}$ of inter-event times obtained from
the sufficient condition
\eqref{eq:linear_case_ETC}.
The blue and red lines depict the bounds in
the cases $G= 1$ and $G=2$, respectively, where
the other parameters of the plant and the controller are set
as in \eqref{eq:constant_parameter}.
In the case $G=2$,
we set $\Gamma = 2.101$ and 
$\gamma = 1$ for
\eqref{eq:M_bound}.
We see from Figure~\ref{fig:ETC_cond} that 
the difference between the cases $G=1$ and $G=2$ is small  and that for example,  
$(\varepsilon,\tau_{\min}) = (0.07,0.001)$ satisfies \eqref{eq:linear_case_ETC}
in both cases.
Therefore, 
the event-triggering mechanism \eqref{eq:ETC_LB} with
$(\varepsilon,\tau_{\min}) = (0.07,0.001)$
achieves exponential stability for both $G=1$ and $G=2$. 
In contrast, the sufficient condition
\eqref{eq:L_thr_cond_STC} does not guarantee that 
the self-triggering mechanism \eqref{eq:STC_Feedback}
constructed for $G=1$
achieves exponential stability for any threshold 
$\varepsilon>0$ in the case $G=2$, since
\eqref{eq:L_thr_cond_STC} is violated for $L=1$.
Note that there is no point in comparing the thresholds $\varepsilon$
between the event-triggering mechanism \eqref{eq:ETC_LB}
and the self-triggering mechanism \eqref{eq:STC_Feedback}.
The event-triggering mechanism \eqref{eq:ETC_LB} measures
$\|x(t_k) - x(t)\|$, whereas the self-triggering
mechanism  \eqref{eq:STC_Feedback} estimates $\|Fx(t_k) - Fx(t)\|$.

Figures~\ref{fig:state_norm_ETC} and \ref{fig:input_ETC}
compare the time responses of 
the event-triggered control system and 
the periodic sampled-data system in the unperturbed case.
We depict
the state norm $\|x(t)\|$ in 
Figure~\ref{fig:state_norm_ETC} and
the input $u(t)$ in Figure~\ref{fig:input_ETC}.
The blue and red lines in these figures are for 
the event-triggered control system
and the periodic sampled-data system, respectively. 
We set the parameters of the plant and the controller 
as in the previous subsection.
The parameters of the event-triggering mechanism \eqref{eq:ETC_LB}
are 
$(\varepsilon,\tau_{\min},\tau_{\max}) = (0.07,0.001,0.5)$.
The sampling period of the periodic system is $h = 0.0205$, which
is the numerically obtained maximum value satisfying the sufficient condition 
\eqref{eq:suff_cond1_PSD} for exponential stability.
We see from Figure~\ref{fig:state_norm_ETC} that 
the state norms of the event-triggered control system
and  the periodic system look almost identical.
In Figure~\ref{fig:updating_interval_ETC}, the blue circles and the red squires depict
the inter-event times of the event-triggered control system and 
the periodic system, respectively.
We see from Figures~\ref{fig:input_ETC} and \ref{fig:updating_interval_ETC} that 
the number of data transmissions in
the even-triggered control system is much smaller 
than that in the periodic system.
This illustrates the effectiveness of event-triggering mechanisms, whose advantage is to change
transmission times depending on the state.

\begin{figure}[tb]
	\centering
	\includegraphics[width = 8.8cm]{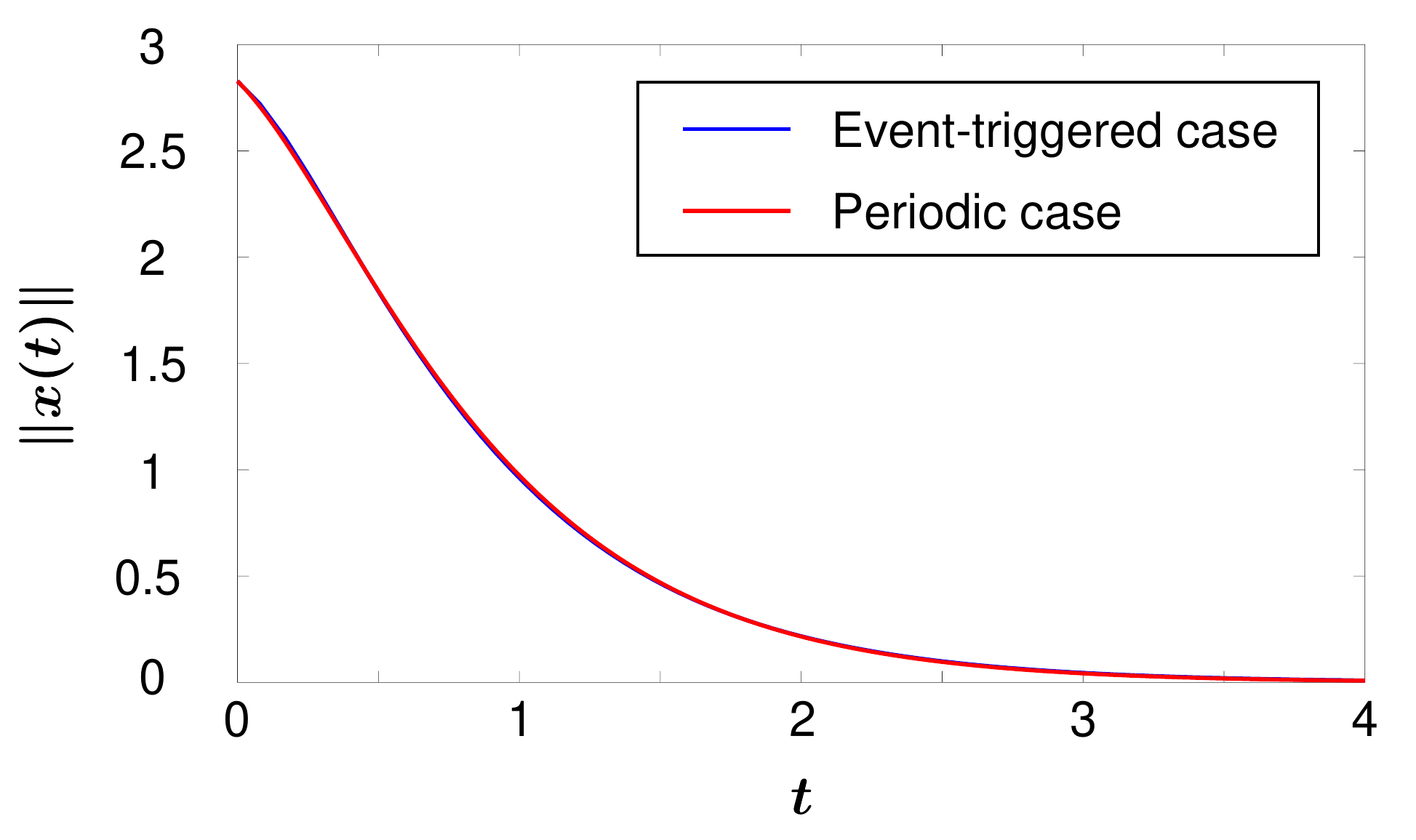}
	\caption{Comparison of state norm $\|x(t)\|$ between 
		event-triggered control system and
		periodic sampled-data system.}
	\label{fig:state_norm_ETC}
\end{figure}
\begin{figure}[tb]
	\centering
	\includegraphics[width =8.8cm]{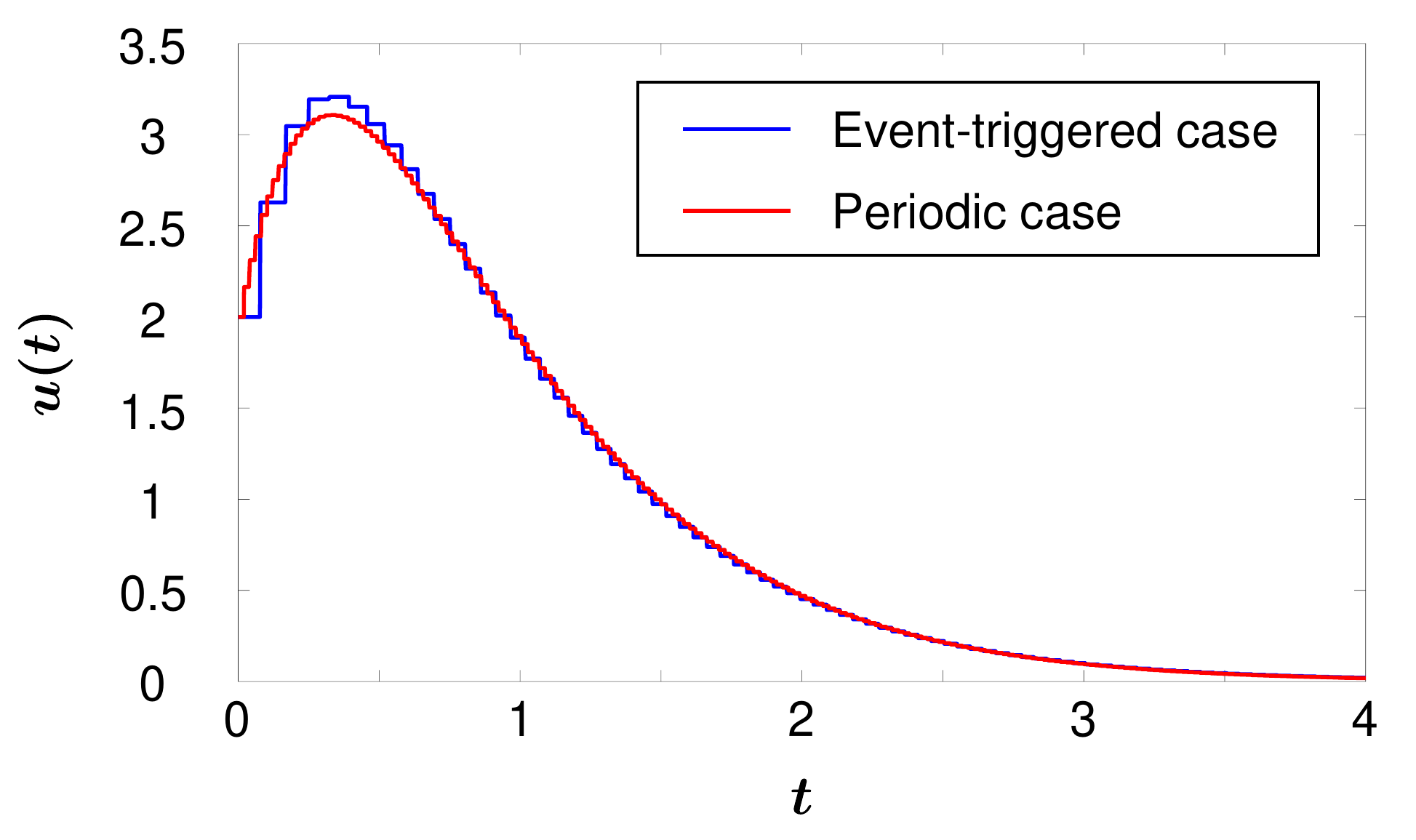}
	\caption{Comparison of input $u(t)$ between event-triggered control system and
		periodic sampled-data system.}
	\label{fig:input_ETC}
\end{figure}

\begin{figure}[tb]
	\centering
	\includegraphics[width = 8.8cm]{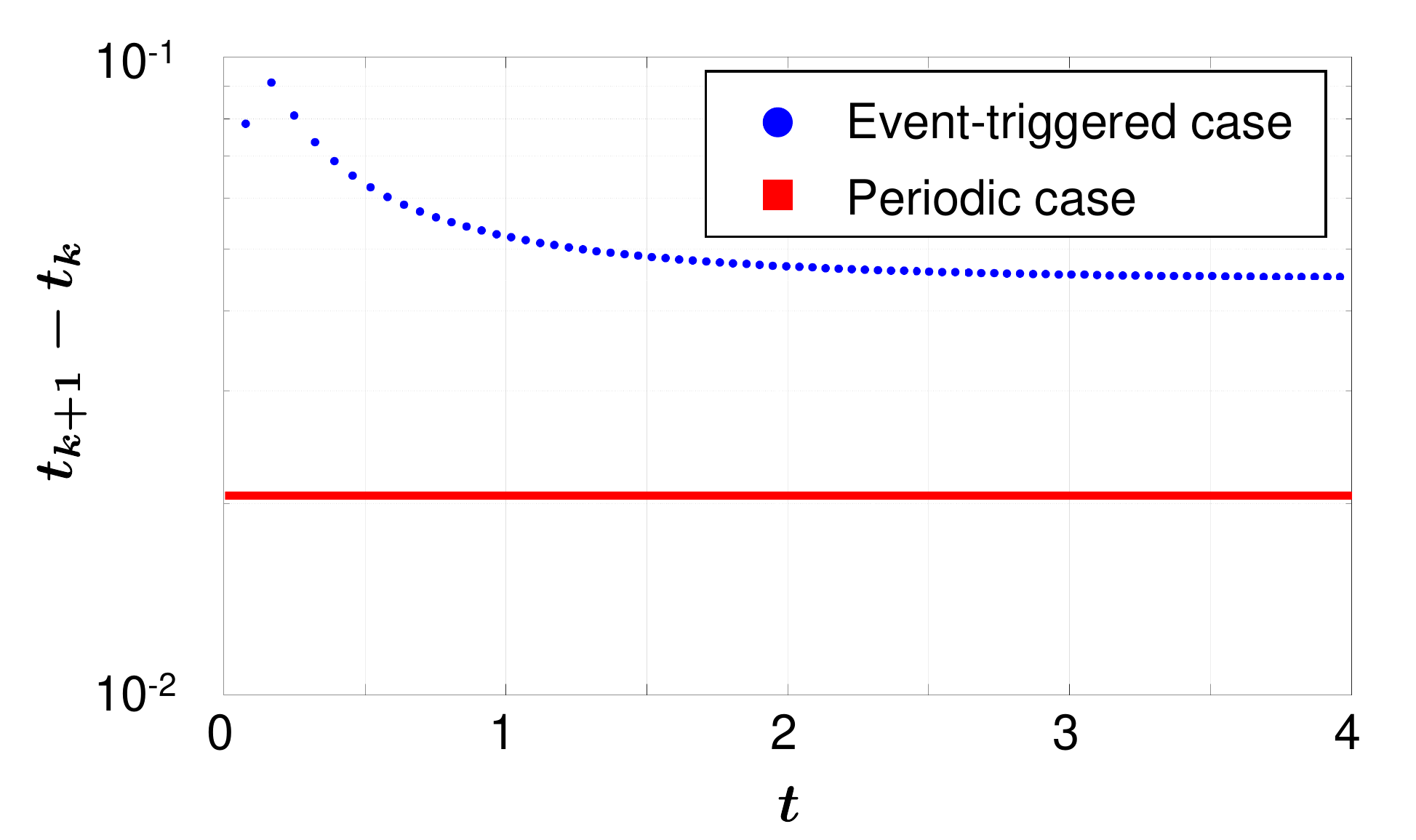}
	\caption{Comparison of inter-event times $t_{k+1}-t_k$ between 
		event-triggered control system and
		periodic sampled-data system.}
	\label{fig:updating_interval_ETC}
\end{figure}

Finally, we make a comparison of
inter-event times between the self-triggering mechanism
\eqref{eq:STC_Feedback} and the event-triggering mechanism 
\eqref{eq:ETC_LB}.
We consider
the self-triggering mechanism
with $(L,\varepsilon,\tau_{\max}) = 
(\sqrt{5}/20, 0.40, 0.5)$, which is used for the small perturbation case in
the previous subsection, and
the event-triggering mechanism 
with $(\varepsilon,\tau_{\min},\tau_{\max}) = 
(0.07, 0.001, 0.5)$ as in the simulations for 
Figures~\ref{fig:state_norm_ETC}--\ref{fig:updating_interval_ETC}.
As shown above,
the lower bounds of inter-event times
are $0.0147$ and $0.001$ for the self-triggering mechanism and
the event-triggering mechanism, respectively.
We see from Figures~\ref{fig:updating_interval}
and \ref{fig:updating_interval_ETC} that 
the inter-event times of the self-triggering mechanism
are larger than those of the event-triggering mechanism.
Hence, the self-triggering mechanism 
is better with respect to inter-event times
than the event-triggering mechanism in this example.
The reason is that the self-triggering mechanism
estimates the input error $Fx(t) - Fx(t_k)$, which
is more essential for control than
the state error $x(t) - x(t_k)$ used in
the event-triggering mechanism.

\section{Periodic event-triggering mechanism for unbounded control}
\label{sec:PET}
In this section, we study event-triggered control for infinite-dimensional systems
with unbounded control operators. 
We consider the evolution equation in the form of \eqref{eq:plant}, but
the difference from the bounded control 
case in Section~\ref{sec:bounded_control}
is that the control operator $B$ 
is unbounded, i.e., $B \in \mathcal{B}(U,X_{-1})$, where
$X_{-1}$ is the extrapolation space of $X$ associated with $T(t)$.
In the  unbounded control case, we cannot apply
standard results on solutions of
evolution equations with Lipschitz perturbations developed in 
Chapter~6 of \cite{Pazy1983}. Therefore, we need to begin by
showing that
the integral equation \eqref{eq:mild_solution}
has a unique solution even in the unbounded control case  and that
this solution satisfies the evolution equation \eqref{eq:plant} interpreted
in $X_{-1}$.
To this end, the compactness of the feedback operator plays an 
important role.
Next, we provide the existence result of periodic event-triggering mechanisms
that achieve exponential stability. Finally, 
we turn back to the bounded control case and give
a simple sufficient condition for the exponential stability
of the periodic event-triggered control system.
\subsection{Solution of evolution equation}
The following properties of $S_\tau$ obtained in 
Lemma 2.2 of \cite{Logemann2003}
are useful
in the analysis of the infinite-dimensional system \eqref{eq:plant} with
an unbounded control operator:
\begin{lemma}[Lemma~2.2 of \cite{Logemann2003}]
	\label{lem:Stau}
	Let $T(t)$ be 
	a strongly continuous semigroup on $X$ and 
	$B \in \mathcal{B}(U,X_{-1})$, where
	$X_{-1}$ is the extrapolation space of $X$ associated with $T(t)$.
	For any $\tau\geq 0$, the operator $S_{\tau}$ defined by 
	\[
	S_{\tau} := \int^{\tau}_0 T(s) Bds 
	\] 
	satisfies $S_{\tau} \in \mathcal{B}(U,X)$
	and
	\[
	\sup_{0\leq t\leq \tau} \|S_{t}\|_{\mathcal{B}(U,X)} < \infty.
	\]
	Moreover, for every $F \in \mathcal{K}(X,U)$,
	\begin{equation}
	\label{eq:SF_limit}
	\lim_{\tau \downarrow 0} \|S_{\tau}F\|_{\mathcal{B}(X)} = 0.
	\end{equation}
\end{lemma}

Using these properties of $S_{\tau}$, we obtain a result on
the existence and regularity of the solution of the integral equation \eqref{eq:mild_solution}.
\begin{theorem}
	\label{prop:existence_solution}
	Let $\tau_{\min} >0$ and
	an increasing sequence $\{t_k\}_{k\in\mathbb{N}_0}$ satisfy $t_0 = 0$ and 
	$t_{k+1}-t_k \geq \tau_{\min}$ for every $k \in \mathbb{N}_0$.
	Assume that $A$ is the generator of a strongly continuous semigroup $T(t)$ on $X$, 
	$B \in \mathcal{B}(U,X_{-1})$, $F \in \mathcal{K}(X,U)$, and a nonlinear operator $\phi:X\to X$ satisfies the Lipschitz condition
	\eqref{eq:Lip_cond}. Then
	the integral equation \eqref{eq:mild_solution} has a unique 
	solution $x$ in $C(\mathbb{R}_+,X)$.
	Furthermore, this solution $x$ satisfies
	\[
	x|_{[t_k,t_{k+1})} \in C^1\big([t_k,t_{k+1}), X_{-1}\big)\qquad \forall k \in \mathbb{N}_0
	\]
	and 
	\[
	\dot x(t) = Ax(t) + BFx(t_k) + \phi\big(
	x(t)
	\big)\qquad \forall t \in (t_k,t_{k+1}),~\forall k \in \mathbb{N}_0,
	\]
	which is interpreted in the extrapolation space $X_{-1}$.
\end{theorem}
By this theorem, 
we say as in the bounded control case that 
the solution of the integral equation \eqref{eq:mild_solution} 
is called a (mild) solution of the evolution equation \eqref{eq:plant}.

Let  $t_1 >0$ be given. We begin by investigating  the integral equation
\begin{equation}
\label{eq:int_eq_x0}
x(t) = T(t) x^0 + \int^t_0
T(t- s) 
\Big(
BF x^0 + \phi\big(x(s) \big)
\Big) ds,~~t \in [0,t_1];\quad x^0 \in X.
\end{equation}
\begin{lemma}
	\label{lem:cont_solution}
	Assume the same hypotheses on $A,B,F,\phi$ as in Theorem~\ref{prop:existence_solution}. Then
	the integral equation \eqref{eq:int_eq_x0} 
	has a unique solution in $C\big([0,t_1],X\big)$.
\end{lemma}
\begin{proof}
	By Corollary~1.3 on p.~185 in \cite{Pazy1983},
	it suffices to show that 
	\[
	\zeta(t) :=
	\int^t_0
	T(t- s) 
	BF x^0 ds = S_t Fx^0
	\]
	satisfies $\zeta \in C\big([0,t_1],X\big)$.

	Let $t \in [0,t_1)$ and $\tau \in (0,t_1-t)$ be given.
	By the strong continuity of $T(t)$,
	there exists $c \geq 1$ such that $\|T(t)\|_{\mathcal{B}(X)}\leq c$ for every $t \in [0,t_1]$.
	Since 
	\begin{align*}
	S_{t+\tau}Fx^0 - S_t Fx^0 = 
	\int^{t+\tau}_t T(s)BFx^0ds 
	=
	T(t) \int^{\tau}_0 T(s)BFx^0 ds 
	= T(t)S_\tau Fx^0,
	\end{align*}
	it follows that 
	\[
	\|S_{t+\tau}Fx^0 - S_t Fx^0\| \leq 
	c \|S_\tau Fx^0\|.
	\]
	Lemma~\ref{lem:Stau} yields
	\[
	\lim_{\tau \downarrow 0}\|\zeta(t+\tau) - \zeta(t) \| = 0,
	\]
	which implies that $\zeta$ is right continuous on $[0,t_1)$.
	Similarly, one can show that $\zeta$ is left continuous on $(0,t_1]$.
	Thus, $\zeta \in C\big([0,t_1],X\big)$. 
\end{proof}

We next study the differentiability of the solution of the integral equation \eqref{eq:int_eq_x0}.
\begin{lemma}
	\label{lem:diff_solution}
	Assume the same hypotheses on $A,B,F,\phi$ as in Theorem~\ref{prop:existence_solution}. 
	The solution $x \in C\big([0,t_1],X\big)$ of the integral equation \eqref{eq:int_eq_x0}
	satisfies \[
	x|_{[0,t_1)} \in C^1\big([0,t_1), X_{-1}\big)
	\]
	and 
	\[
	\dot x(t) = Ax(t) + BFx(t_k) + \phi\big(
	x(t)
	\big)\qquad \forall t \in [0,t_1),
	\]
	which is interpreted in the extrapolation space $X_{-1}$.
\end{lemma}
\begin{proof}
	Define 
	$g_1(t) :=BFx^0 $ and $g_2(t) := \phi\big(x(t) \big)$ for $t \in [0,t_1)$.
	By Theorem~2.4 on p.~107 in \cite{Pazy1983},
	it is enough to show that for each $i \in \{1,2\}$, 
	$g_i \in C \big([0,t_1),X_{-1}\big)$ and $v_i$ defined by
	\[
	v_i(t) := \int^t_0 T(t-s)g_i(s)ds,\quad t \in [0,t_1)
	\]
	satisfies  $v_i \in X$ for every $t \in [0,t_1)$
	and $Av_i \in C \big([0,t_1),X_{-1}\big)$.
	
	Clearly, the constant function $g_1$ belongs to $C \big([0,t_1),X_{-1}\big)$.
	Since
	\[
	v_1(t) = S_t Fx^0,
	\]
	it follows from Lemma~\ref{lem:Stau} 
	that $v_1(t) \in X$ for every $t \in [0,t_1)$.
	Moreover,
	\[
	Av_1(t) = (T(t) - I)BFx^0,
	\]
	and hence $Av_1 \in C \big([0,t_1),X_{-1}\big)$
	by 
	the strong continuity of $T(t)$.
	
	Let us next investigate $g_2$ and $v_2$.
	Since $x \in C \big([0,t_1],X\big)$ and $\phi$ is Lipschitz continuous on $X$, it follows that 
	$g_2 \in C \big([0,t_1),X\big)$. 
	Let $\lambda \in \varrho(A)$  and $(X_{-1},\|\cdot\|_{X_{-1}})$ be the completion
	of $(X, \|\cdot \|_{-1})$, where $\|x\|_{-1} := \|(\lambda I - A)^{-1}x\|$ for $x \in X$.
	Since
	\begin{equation}
	\label{eq:cont_emb}
	\|x\|_{X_{-1}} = \|x\|_{-1} \leq \|(\lambda I - A)^{-1}\|_{\mathcal{B}(X)}~\! \|x\|\qquad \forall x \in X,
	\end{equation}
	it follows that  $g_2 \in C \big([0,t_1),X_{-1}\big)$.
	By definition, $v_2(t) \in X$ for every $t \in [0,t_1)$.
	To show $Av_2 \in C \big([0,t_1),X_{-1}\big)$, 
	it is enough to prove $Ag_2  \in C \big([0,t_1),X_{-1}\big)$, because
	\[
	Av_2(t) = \int^t_0 T(t-s)Ag_2(s)ds;
	\]
	see, e.g., Proposition~1.3.4 on p.~24 in \cite{Arendt2001} for
	the continuity property of convolutions.

	Since $\lambda I -A$ is an isometry from $X$ to $X_{-1}$ (see, e.g., Theorem~II.5.5 on p. 126 in \cite{Engel2000}),
	it follows that for every $t,s \in [0,t_1)$,
	\begin{align}
	\label{eq:Af2_bound}
	\|Ag_2(t) - Ag_2(s)\|_{X_{-1}}
	&\leq
	\big\|
	\phi\big(x(t) \big) - \phi\big(x(s) \big) 
	\big\| + 
	|\lambda| ~\!\big \|\phi\big(x(t) \big) - \phi\big(x(s) \big) \big\|_{X_{-1}}.
	\end{align}
	Using 
	$x \in C \big([0,t_1],X\big)$, the Lipschitz continuity of 
	$\phi$ on $X$, and \eqref{eq:cont_emb},
	we obtain $Ag_2  \in C \big([0,t_1),X_{-1}\big)$.
	This completes the proof.
\end{proof}

\begin{proof}[Proof of Theorem~\ref{prop:existence_solution}]
	Since $x(t_k) \in X$ for every $k \in \mathbb{N}_0$ by Lemma~\ref{lem:cont_solution},
	we obtain the desired conclusion by
	repeating the argument in Lemmas~\ref{lem:cont_solution} and \ref{lem:diff_solution}.
\end{proof}

\subsection{Periodic event-triggering mechanism}
\label{sec:unbounded_control}
We define the increasing sequence $\{t_k\}_{k\in\mathbb{N}_0}$ by
\begin{subequations}
	\label{eq:PET}
	\noeqref{eq:PET1, eq:PET2}
	\begin{align}
	&t_{k+1} := 
	\min\{ t_k + \ell_{\max}h,~\bar t_{k+1}
	\};\quad  \label{eq:PET1}t_0 := 0
	\\
	&\bar t_{k+1} := \min \big\{ \ell h > t_k:
	\|x(\ell h) - x(t_k)\| > \varepsilon \|x(t_k)\|,~\ell \in \mathbb{N}
	\big\},	\quad  k \in \mathbb{N}_0, \label{eq:PET2}
	\end{align}
\end{subequations}
where $h>0$
is a sampling period, 
$\varepsilon \geq 0$ is a threshold parameter, 
and $\ell_{\max} \in \mathbb{N}$ determines an upper bound
of inter-event times as follows: $t_{k+1} - t_k \leq \ell_{\max} h$
for every $k \in \mathbb{N}_0$.
We call \eqref{eq:PET}  a {\em periodic event-triggering mechanism}
\cite{Heemels2013}.
In the case $\varepsilon = 0$, the state $x(t)$ is transmitted at every $t = kh$, $k \in \mathbb{N}_0$, unless $x(kh+h) = x(kh)$. 
Therefore, the periodic event-triggering mechanism \eqref{eq:PET} with $\varepsilon=0$
can be regarded as the conventional periodic sampling process.
The periodic event-triggering mechanism \eqref{eq:PET} 
checks the condition only periodically
unlike the event-triggering mechanism \eqref{eq:ETC_LB}.
This discrete behavior may degrade the control performance  for
a large $h$,
but it makes the periodic event-triggering mechanism \eqref{eq:PET} 
better suited for practical implementations.

We analyze the periodic event-triggered control system by discretizing
the closed-loop system with period $h$.
The resulting discrete-time system has a bounded control operator by Lemma~\ref{lem:Stau}.
Combining this with an estimate of the perturbation term by
Gronwall's inequality,
we obtain a sufficient condition for exponential stability.
\begin{lemma}
	\label{lem:PET_stability}
		Assume that $A$ generates a strongly continuous semigroup $T(t)$ on $X$, 
	$B \in \mathcal{B}(U,X_{-1})$, $F \in  \mathcal{K}(X,U)$, 
	and a nonlinear operator $\phi:X\to X$ satisfies the Lipschitz condition \eqref{eq:Lip_cond}.
	Moreover, assume that 
	$\Delta_h$ defined by \eqref{eq:Delta_def} is power stable for some $h>0$, i.e.,
	there exist $\Omega \geq 1$ and $\omega \in (0,1)$ such that
	\[
	\|\Delta_h^k\|_{\mathcal{B}(X)} \leq \Omega \omega^k\qquad \forall k \in \mathbb{N}_0.
	\]
	Then the event-triggered control system \eqref{eq:plant} with the 
	mechanism \eqref{eq:PET} is exponentially stable
	for every $\ell_{\max} \in \mathbb{N}$ if
	 $L, \varepsilon \geq 0$ satisfy
	\begin{equation}
	\label{eq:PET_eps_cond}
	\varepsilon \Omega\big( \|S_hF\|_{\mathcal{B}(X)} + c_3(e^{c_2L h} - 1) \big)< 
	1 - \omega - c_1\Omega(e^{c_2L h} - 1),
	\end{equation}
	where 
	\begin{align}
	\label{eq:c_def}
	c_1 := 
	\sup_{0\leq \tau \leq h} \|\Delta_\tau\|_{\mathcal{B}(X)},~~
	c_2  := \sup_{0\leq \tau \leq h} \|T(\tau)\|_{\mathcal{B}(X)},~~
	c_3 := \sup_{0\leq \tau \leq h} \|S_{\tau} F\|_{\mathcal{B}(X)}.
	\end{align}
\end{lemma}
\begin{proof}
	As in the proof of Theorem~5.8 in \cite{Wakaiki2018_EVC},
	define a new norm $|\cdot|_{\rm d}$ on $X$ by 	
	\begin{align}
	\label{eq:new_norm}
	|x|_{\rm d} := \sup_{\ell \in \mathbb{N}_0} \|\omega^{-\ell} \Delta_h^\ell x\|,
	\quad x \in X.
	\end{align}
	Similarly to the norm $|\cdot|$ defined by \eqref{eq:new_norm_cont},
	the discrete-time counterpart $|\cdot|_{\rm d}$ 
	satisfies
	\begin{equation}
	\label{eq:new_norm_prop}
	\|x\| \leq |x|_{\rm d} \leq \Omega \|x\|,\quad |\Delta_h^k x|_{\rm d} \leq \omega^k |x|_{\rm d}\qquad \forall x \in X,~\forall k \in \mathbb{N}_0.
	\end{equation}
	
	For the time sequence $\{t_k\}_{k\in \mathbb{N}_0}$ defined by \eqref{eq:PET},
	let $\ell_k  \in \mathbb{N}_0$ satisfy $t_k = \ell_k h$.
	The error $e$ induced by the event-triggering implementation is given by
	\[
	e(\ell h) := x(\ell_k h) - x(\ell h),\quad  \ell \in [\ell_k,\ell_{k+1}) \cap \mathbb{N}_0,~ k \in \mathbb{N}_0.
	\]
	Under the periodic event-triggering mechanism \eqref{eq:PET}, 
	the error $e$ satisfies
	\begin{equation}
	\label{eq:error_bound_PET}
	\|e(\ell h)\| \leq \varepsilon \|x(\ell_k h)\|\qquad \forall \ell \in [\ell_k,\ell_{k+1})\cap \mathbb{N}_0,~\forall k \in \mathbb{N}_0.
	\end{equation}
	The solution of  the integral equation \eqref{eq:mild_solution} can be rewritten as
	\begin{equation}
	\label{eq:dist_state_eq}
	x(\ell h + \tau ) = \Delta_\tau x(\ell h )  + \int^\tau_0 T(\tau-s)\phi\big(
	x(\ell h+s)
	\big)ds + S_\tau Fe (\ell h)
	\end{equation}
	for every $\tau \in (0,h]$ and $\ell \in \mathbb{N}_0$.
Gronwall's inequality yields 
\begin{align}
\|x(\ell h + \tau ) \|
&\leq 
\big(
c_1\|x(\ell h)\| + \varepsilon c_3 \|x(\ell_k h)\|	
\big)
+ c_2 L \int^\tau_0 \|x(\ell h +s)\| ds \notag \\
&\leq 
e^{c_2L \tau}\big(
c_1\|x(\ell h)\| + \varepsilon c_3 \|x(\ell_k h)\|	
\big)  \label{eq:inter_sample_bound}
\end{align}
for every $\tau \in (0,h]$, $\ell \in [\ell_k,\ell_{k+1})\cap \mathbb{N}_0$, and $k \in \mathbb{N}_0$, where
$c_1,c_2,c_3\geq0$ are defined by \eqref{eq:c_def}.
It follows from \eqref{eq:dist_state_eq} with $\tau = h$ that
\begin{align}
\big|
x\big((\ell+1 )h \big)
\big|_{\rm d} &\leq 
\omega |x(\ell h)|_{\rm d} 
+ \varepsilon \Omega \|S_hF\|_{\mathcal{B}(X)}  |x(\ell_k h)|_{\rm d} \notag \\
&\qquad + c_2 L \Omega \int^h_0 e^{c_2L s} ds 
\big(
c_1|x(\ell h)|_{\rm d} 
+ \varepsilon c_3 |x(\ell_k h)|_{\rm d}	
\big) \notag \\
&\leq 
\widetilde \omega_1 (L) |x(\ell h)|_{\rm d} + 
\delta_1(L,\varepsilon) |x(\ell_k h)|_{\rm d},
\label{eq:sample_time_bound}
\end{align}
where
\[
\widetilde \omega_1 (L) := \omega + c_1\Omega(e^{c_2L h} - 1),\quad
\delta_1(L,\varepsilon) :=
\varepsilon\Omega
\big( \|S_hF\|_{\mathcal{B}(X)} + c_3(e^{c_2L h} - 1)\big).
\]
Proceeding by induction, 
we have
	\begin{align*}
	\big|x(\ell_{k+1} h)\big|_{\rm d} &\leq 
	\widetilde \omega_1 (L)^{p_k} |x(\ell_k h)|_{\rm d} + 
	\frac{1-\widetilde \omega_1 (L)^{p_k}}{1-\widetilde \omega_1 (L)}
	\delta_1(L,\varepsilon) |x(\ell_k h)|_{\rm d}  \\
	&= \big(
	\widetilde \omega_1 (L)^{p_k} [1-\delta_2(L,\varepsilon)] + \delta_2(L,\varepsilon)
	\big) |x(\ell_k h)|_{\rm d}\qquad \forall k \in \mathbb{N}_0,
	\end{align*}
	where 
	\[
	p_{k} := \ell_{k+1} - \ell_k,\quad 
	\delta_2(L,\varepsilon) := \frac{\delta_1(L,\varepsilon)}
	{1-\widetilde \omega_1 (L)}.
	\]
	For $L,\varepsilon \geq 0$, \eqref{eq:PET_eps_cond}
	holds if and only if $ \widetilde \omega_1 (L) < 1$ and
	$ \delta_2(L,\varepsilon)  < 1$.
	Let $L,\varepsilon \geq 0$ satisfy
	\eqref{eq:PET_eps_cond}, and define $\widetilde \omega := \widetilde \omega_1 (L)$ and
	$\delta := \delta_2(L,\varepsilon) $.
	If we define the function $f_{\rm p}$ on $\mathbb{N}$ by
	\[
	f_{\rm p}(\ell) := \frac{-\log (\widetilde \omega^\ell(1-\delta) + \delta)}{\ell h},
	\] 
	then
	$f_{\rm p}$ is positive and monotonically decreasing on $\mathbb{N}$.
	Therefore, 
	\[
	\big|x(\ell_{k+1} h)\big|_{\rm d}  
	\leq e^{-\gamma_0 p_kh } |x(\ell_k h)|_{\rm d} \qquad \forall k \in \mathbb{N}_0,
	\]
	where $\gamma_0 := f_{\rm p}(\ell_{\max}) > 0$.
	By induction, we obtain
	\[
	|x(\ell_kh)|_{\rm d} \leq e^{-\gamma_0 \ell_k h }|x^0|_{\rm d}\qquad 
	\forall x^0 \in X,~\forall k\in \mathbb{N}_0.
	\]
	Using \eqref{eq:inter_sample_bound} and \eqref{eq:sample_time_bound}
	again, 
	we obtain
	\begin{align*}
	\|x(\ell_kh+\tau)\| 
&\leq M e^{-\gamma_0 \ell_k h} \|x^0\|\\
	&\leq 
	(M e^{\gamma_0 \ell_{\max}h}) e^{-\gamma_0(\ell_kh+\tau)} \|x^0\|
	\qquad \forall \tau \in [0,p_k h],~\forall k \in \mathbb{N}_0
	\end{align*}
	for some $M >0$.
	Thus, the event-triggered control system is exponentially stable.
\end{proof}

Define an operator $A_{BF}$ on $X$ by
\begin{equation}
\label{eq:ABF_def}
A_{BF} x := (A+BF)x\quad \text{with~~}\dom (A_{BF}) := \{
x \in X : (A+BF) x \in X
\},
\end{equation}
which we distinguish from the unbounded operator $A+BF$ on $X_{-1}$ with 
$\dom(A+BF) = X$. 
Under the assumption that $T(t)$ is analytic, 
Theorem~4.8 in \cite{Logemann2003} shows that
the exponential stability of linear periodic sampled-data systems
is robust with respect to sampling.
\begin{theorem}[Theorem~4.8 in \cite{Logemann2003}]
	\label{thm:SD_stability}
	Assume that $A$ generates an analytic semigroup $T(t)$ on $X$, 
	$B \in \mathcal{B}(U,X_{-1})$, and $F \in \mathcal{K}(X,U)$.
	Moreover, assume that the semigroup generated by $A_{BF}$ given 
	in \eqref{eq:ABF_def}
		is exponentially stable.
	Then
	there exists $h^* >0$ such that for every $h \in (0,h^*)$,
	the linear periodic sampled-data  system \eqref{eq:plant} with $\phi \equiv 0$ and
	$t_{k+1}-t_k \equiv h$ is exponentially stable.
\end{theorem}

See also \cite{Rebarber2006} for another result on robustness of stabilization with respect to sampling
in the unbounded control case.

Combining Lemma~\ref{lem:PET_stability} and Theorem~\ref{thm:SD_stability},
we obtain a result on the existence of a 
periodic event-triggering mechanism
that achieves exponential stability.
\begin{theorem}
	\label{thm:PETC_LB_stability}
	Assume the same hypotheses on $A,B,F,T(t)$ as in Theorem~\ref{thm:SD_stability}, and
	choose $h >0$ so that the  linear 
	periodic sampled-data system \eqref{eq:plant} 
	with $\phi \equiv  0$  and $t_{k+1}-t_k \equiv h$ is exponentially stable.
	Moreover, assume that  a nonlinear operator $\phi:X\to X$ satisfies
	the Lipschitz condition \eqref{eq:Lip_cond}.
	Then there exist  $\varepsilon^* =\varepsilon^*(h)> 0 $ and $L^* =L^*(h)>0$ such that 
	for every $\varepsilon \in [0,\varepsilon^*]$ and every $L \in [0,L^*]$, 
	the system \eqref{eq:plant} with the
	periodic event-triggering mechanism \eqref{eq:PET} is exponentially stable.
\end{theorem}

\begin{proof}
	By Lemma~2.3  in \cite{Logemann2003}, the linear periodic sampled-data system \eqref{eq:plant} 
	with $\phi \equiv 0$ and $t_{k+1}-t_k \equiv h$ 
	is exponentially stable if and only if
	the operator $\Delta_h$ is power stable.
	Combining this with Lemma~\ref{lem:PET_stability}, 
	we obtain the desired result.
\end{proof}

Finally, we return to the bounded control case
$B \in \mathcal{B}(U,X)$.
Suppose that the semigroup $T_{BF}(t)$ is exponentially
stable, i.e, 
\eqref{eq:TBF_exp_stable} holds for some $\Gamma\geq 1$ and $\gamma >0$.
For $h >0$, define
\begin{equation}
\label{eq:W_def}
W(h) :=
\Gamma e^{\gamma h}\sup_{0 \leq t\leq h} \big\|
T(h-t) BF [I-T_{BF}(t)]\big\|_{\mathcal{B}(X)}.
\end{equation}
As explained in Remark \ref{rem:periodic_case},
$\Delta_h$ defined by \eqref{eq:Delta_def} is power stable
if $W(h) < \gamma$.
In such a case, we obtain
\[
\| \Delta_h^k \|_{\mathcal{B}(X)} \leq \Gamma \omega^k\qquad \forall k \in \mathbb{N}_0,
\]
where
$
\omega := 1 - [\gamma - W(h)] e^{-\gamma h}h
$; see (3.6) in the proof of Theorem~3.1 of \cite{Logemann2003}.
This fact, together with Lemma~\ref{lem:PET_stability}, yields a 
simple sufficient condition for 
the periodic event-triggered control system to be exponentially stable.
To avoid the trivial case in which
the open-loop system is exponentially stable,
we here additionally assume that $S_hF \not = 0$ holds for every $h>0$ satisfying $W(h) < \gamma$.
\begin{corollary}
	Suppose that Assumption~\ref{assump:bounded_case} is satisfied and 
	 $S_hF \not = 0$ holds for every $h>0$ satisfying $W(h) < \gamma$.
	The system \eqref{eq:plant} with the
periodic event-triggering mechanism \eqref{eq:PET} is exponentially stable
for every $\ell_{\max} \in \mathbb{N}$ if
$h>0$ and  $L, \varepsilon \geq 0$ satisfy
\begin{equation}
\label{eq:PETC_cond_bounded}
\varepsilon < \frac{
[\gamma - W(h)]e^{-\gamma h}h - c_1 \Gamma(e^{c_2L h} - 1)}{\Gamma \big(\|S_hF\|_{\mathcal{B}(X)} + c_3(e^{c_2L h} - 1) \big)},
\end{equation}
where  $c_1,c_2,c_3\geq0$ are defined by \eqref{eq:c_def}.
\end{corollary}

For the numerical example 
of the unperturbed case $\psi(x_2) = x_2$ in Section~\ref{sec:numerical_example},
one can observe that 
the bounds of the parameters $(\varepsilon , h)$
satisfying \eqref{eq:PETC_cond_bounded} are similar to
to those of $(\varepsilon , \tau_{\min})$ shown in Figure~\ref{fig:ETC_cond}.
Moreover, as expected easily, if $h$ is small, then 
numerical simulations of the periodic event-triggered control systems
are also closely similar to those in Figures~\ref{fig:state_norm_ETC}--\ref{fig:updating_interval_ETC}.
We omit these figures because they show 
almost identical trends as Figures~\ref{fig:ETC_cond}--\ref{fig:updating_interval_ETC}.

\section{Conclusion}
In this paper,
we have analyzed the exponential stability 
of infinite-dimensional event/self-triggered control systems with Lipschitz perturbations.
The fundamental assumption is that the feedback operator is compact,
which guarantees the strict positiveness of inter-event times and
the existence of the mild solution of the evolution equation
with an unbounded control operator.
We have shown that 
if the parameters of the event/self-triggering mechanisms are appropriately 
chosen,
then
exponential stability is preserved under all
perturbations with sufficiently small Lipschitz constants.
Moreover, in the bounded control case, we have provided 
simple sufficient conditions for exponential stability.

\providecommand{\href}[2]{#2}
\providecommand{\arxiv}[1]{\href{http://arxiv.org/abs/#1}{arXiv:#1}}
\providecommand{\url}[1]{\texttt{#1}}
\providecommand{\urlprefix}{URL }

\medskip
Received xxxx 20xx; revised xxxx 20xx.
\medskip

\end{document}